\documentclass[11pt]{amsart}
\usepackage{amssymb, amsmath, amscd, eucal, rotating}

\input xy
 \xyoption{all}

\numberwithin{equation}{section}

\theoremstyle{plain}

\newtheorem{proposition}{Proposition}[section]
\newtheorem{theorem}[proposition]{Theorem}		
\newtheorem{corollary}[proposition]{Corollary}
\newtheorem{lemma}[proposition]{Lemma}

\theoremstyle{definition}

\newtheorem{remark}[proposition]{Remark}

\newcommand{\CBbb}{\mathbb C}
\newcommand{\RBbb}{\mathbb R}
\newcommand{\ZBbb}{\mathbb Z}

\newcommand{\HBbb}{\mathbb H}

\newcommand{\BBbb}{\mathbb B}

\newcommand{\YMH}{\mathop{\rm YMH}\nolimits}

\newcommand{\Sym}{S}

\newcommand{\Hom}{\mathop{\rm Hom}\nolimits}

\newcommand{\R}{{\mathfrak R}}
\newcommand{\X}{{\mathfrak X}}

\newcommand{\Hi}{{\mathcal H}}

\newcommand{\Aut}{\mathop{\rm Aut}\nolimits}
\newcommand{\Map}{\mathop{\rm Map}\nolimits}

\newcommand{\hol}{\mathop{\rm hol}\nolimits}
\newcommand{\ad}{\mathop{\rm ad}\nolimits}

\newcommand{\tr}{\mathop{\rm Tr}\nolimits}
\newcommand{\dbar}{\bar\partial}
\newcommand{\lra}{\longrightarrow}

\newcommand{\G}{\mathcal G}
\newcommand{\B}{\mathcal B}

\newcommand{\A}{\mathcal A}

\newcommand{\LL}{\mathcal L}

\newcommand{\M}{{\mathfrak M}}

\newcommand{\N}{{\mathfrak N}}

\newcommand{\mcg}{{\rm Mod}}
\newcommand{\Tor}{{\mathcal I}}
\newcommand{\PTor}{{\mathcal P}{\mathcal I}}

\newcommand{\SL}{\mathsf{SL}}
\newcommand{\PSL}{\mathsf{PSL}}
\newcommand{\GL}{\mathsf{GL}}

\newcommand{\SU}{\mathsf{SU}}
\newcommand{\PU}{\mathsf{PU}}
\newcommand{\U}{\mathsf{U}}

\newcommand{\Sp}{\mathsf{Sp}}

\newcommand{\Diff}{\mathsf{Diff}}

\newcommand{\doubleslash}{\bigr/ \negthinspace\negthinspace \bigr/}

\begin{document}


\title[$\SL(2,\CBbb)$ character varieties of surface groups]
	{Cohomology of  $\SL(2,\CBbb)$ character varieties of surface groups and the action of the Torelli group}

\author[Daskalopoulos]{Georgios D. Daskalopoulos}

\address{Department of Mathematics \\
		Brown University \\
		Providence,  RI  02912}

\thanks{G.D. supported in part by NSF grant DMS-0604930}

\email{daskal@math.brown.edu}

\author[Wentworth]{Richard A. Wentworth}

   

\address{Department of Mathematics \\
   University of Maryland \\
   College Park, MD 20742}
 
\email{raw@umd.edu}

\author[Wilkin]{Graeme Wilkin}

\address{Department of Mathematics \\
University of Colorado at Boulder \\
Boulder, CO 80309}

\email{graeme.wilkin@colorado.edu}

\thanks{R.W. supported in part by NSF grant DMS-0805797}

\date{\today}

\begin{abstract} We determine the action of the Torelli group  on the equivariant cohomology of the space of flat $\SL(2,\CBbb)$ connections on a closed Riemann surface.  We show that the trivial part of the action contains the equivariant cohomology of the even component of the space of flat $\PSL(2,\CBbb)$ connections.  The non-trivial part  consists of the even alternating products of degree two Prym representations, so that the kernel of the action is precisely the Prym-Torelli group.  We compute the Betti numbers of the ordinary cohomology of the moduli space of flat $\SL(2,\CBbb)$ connections. Using results of Cappell-Lee-Miller we show that the Prym-Torelli group, which acts trivially on equivariant cohomology,  acts non-trivially on ordinary cohomology.
\end{abstract}



\maketitle

\thispagestyle{empty}


\baselineskip=16pt
\setcounter{footnote}{0}


\section{Introduction}

The Torelli group acts trivially on the equivariant cohomology of the space of flat unitary connections  on a Riemann surface.  
This follows from the fact that the inclusion of the subset of flat connections into the space of all unitary connections induces a surjection on equivariant cohomology  (see \cite{AB, D, R} and Theorem \ref{T:eq2} below). 
 The latter result may be viewed as an infinite dimensional analogue of a general theorem on   symplectic quotients that has become known as \emph{Kirwan surjectivity} (see  \cite{Ki}). 
 The moduli space of flat $\SL(2,\CBbb)$ connections  has a  gauge theoretic construction due to 
Hitchin (see \cite{H}).  A recent result \cite{DWWW} shows that Kirwan surjectivity actually 
 fails in this case.  In this paper, we show that this failure is detected by the action 
of the Torelli group.
 
To state the results more precisely, recall the notion of a character variety (see \cite{CS,LM} for background). Let $M$ be a closed oriented surface of genus $g\geq 2$, and let
$p\in M$ be a point which will remain fixed throughout.  We set $\pi=\pi_1(M,p)$. 
Let  $\Hom(\pi,\SL(2,\CBbb))$ denote the set of homomorphisms from $\pi$ to $\SL(2,\CBbb)$. This has the structure of an affine algebraic variety.
Let
\begin{equation*} 
\X_0(\pi) =\Hom(\pi,\SL(2,\CBbb))\doubleslash \SL(2,\CBbb)
\end{equation*}
denote the character variety,
where the double slash indicates the invariant theoretic quotient by   overall conjugation of $\SL(2,\CBbb)$.  Then $\X_0(\pi)$ is an irreducible affine variety of complex dimension $6g-6$.
There is a surjective algebraic quotient map $\Hom(\pi,\SL(2,\CBbb)) \to \X_0(\pi)$,  and this is a geometric quotient on the open set of irreducible (or simple) representations. Points of $\X_0(\pi)$ are in 1-1 correspondence with conjugacy classes of semisimple (or reductive) representations, and every $\SL(2,\CBbb)$ orbit in $\Hom(\pi,\SL(2,\CBbb))$ contains a semisimple representation in its closure.
The mapping class group $\mcg(M)$ is the group of components of orientation preserving diffeomorphisms of $M$.  Since $\mcg(M)$ acts by outer automorphisms of $\pi$, there is a naturally induced action on  $\X_0(\pi)$, and hence also on the cohomology $H^\ast(\X_0(\pi))$.\footnote{Unless otherwise stated, cohomology will always be taken with rational coefficients.}   
Since $\X_0(\pi)$ is a categorical quotient
it is also natural to  consider the $\SL(2,\CBbb)$-equivariant cohomology 
\begin{equation} \label{E:equivariant}
H^\ast_{eq.}(\X_0(\pi)) :=H^\ast_{\SL(2,\CBbb)}(\Hom(\pi,\SL(2,\CBbb)))
\end{equation}
	  With a slight abuse of terminology we will often  refer to $H^\ast_{eq.}(\X_0(\pi))$ as the \emph{equivariant cohomology of} $\X_0(\pi)$.
The action of $\Aut(\pi)$  on $H^\ast_{eq.}(\X_0(\pi))$ in fact descends to an action of $\mcg(M)$ (inner automorphisms act trivially; see Section \ref{S:action}).  

Next, let
 \begin{equation} \label{E:gamma}
 \Gamma_2 =H^1(M, \ZBbb/2)\simeq\Hom(\pi,\{\pm 1\})
 \end{equation}
  Then $\Gamma_2$ acts on $\Hom(\pi,\SL(2,\CBbb))$  by  $(\gamma\rho)(x)=\gamma(x)\rho(x)$.
    This action  commutes with conjugation by $\SL(2,\CBbb)$, and hence it  defines an action on 
    $\X_0(\pi)$ and on the  ordinary and equivariant cohomologies.  We denote the $\Gamma_2$ invariant parts of the cohomology by $H^\ast(\X_0(\pi))^{\Gamma_2}$ and
  $H^\ast_{eq.}(\X_0(\pi))^{\Gamma_2}$.

    The Torelli group $\Tor(M)$ is the subgroup of $\mcg(M)$ that acts trivially on the homology of $M$. 
    In particular,  the action of $\Gamma_2$  commutes with the action of $\Tor(M)$.
     The kernel of $\gamma \in \Gamma_2\simeq\Hom(\pi,\{\pm 1\})$, $\gamma\neq 1$, defines an unramified double cover $M_\gamma\to M$ with involution $\sigma$.  Let $W^+_\gamma$ (resp.\ $W^-_\gamma$) denote the  $2g$ (resp.\ $2g-2$) dimensional   $+1$ (resp.\ $-1$) eigensubspaces  of $H^1( M_\gamma)$ for $\sigma$.  A lift of a diffeomorphism of $M$ representing an element of $\Tor(M)$ that commutes with $\sigma$ may or may not be in the Torelli group of $ M_\gamma$; although it acts trivially on $W^+_\gamma$ it may act non-trivially on $W^-_\gamma$.  Since the two lifts differ by $\sigma$, there is thus defined a representation
\begin{equation} \label{E:prym}
\Pi_\gamma : \Tor(M)\lra \Sp(W^-_\gamma, \ZBbb)\bigr/\{\pm I\}
\end{equation}
which is called the (degree 2) \emph{Prym representation} of $\Tor(M)$ associated to $\gamma$.  An element in $\ker \Pi_\gamma$ has a lift which lies in $\Tor( M_\gamma)$.
By a  theorem of Looijenga \cite{Lo}, the image of $\Pi_\gamma$ has finite index for $g>2$.  Note that the representations for various $\gamma\neq 1$ are isomorphic via outer automorphisms of $\Tor(M)$.
 $\Pi_\gamma$ induces non-trivial representations of $\Tor(M)$ on the exterior products
\begin{equation} \label{E:altprym}
V(q,\gamma)=\Lambda^q W^-_\gamma
\end{equation}
when $q$ is even.
Finally, we define the (degree 2) \emph{Prym-Torelli group}
\begin{equation} \label{E:prymtorelli}
\PTor(M)=\bigcap_{1\neq \gamma\in \Gamma_2} \ker \Pi_\gamma
\end{equation}

With this background, we may summarize the first result of this paper  as follows:

\begin{theorem} \label{T:main}
 \begin{enumerate}
 \item $\Tor(M)$ acts trivially on $H^\ast_{eq.}(\X_0(\pi))^{\Gamma_2}$.
 \item For $q\in S=\{2j\}_{j=1}^{g-2}$  the action of $\Tor(M)$ splits as
 $$
 H^{6g-6-q}_{eq.}(\X_0(\pi))=H^{6g-6-q}_{eq.}(\X_0(\pi))^{\Gamma_2}\oplus \bigoplus_{1\neq\gamma\in \Gamma_2} V(q,\gamma)
$$
  In particular, $\PTor(M)$ acts trivially  and $\Tor(M)$ acts non-trivially  on  $H^\ast_{eq.}(\X_0(\pi))$ for $g>2$. The splitting of the sum of $V(q,\gamma)$'s is canonically determined by a choice of homology basis of $M$.
 \item $\Tor(M)$ acts trivially on $ H^{6g-6-q}_{eq.}(\X_0(\pi))$ for $q\not\in S$.
\end{enumerate}
\end{theorem}

The proof of Theorem \ref{T:main} uses the singular version of infinite dimensional Morse theory developed in \cite{DWWW} to build the equivariant cohomology from a Morse-Bott type stratification. 
We will view $H^\ast_{eq.}(\X_0(\pi))$ via gauge theory as follows.
Let $\B^{ss}_0(2,0)$ denote the space of semistable Higgs bundles of rank 2 with fixed trivial determinant on $M$, let $\G_0$ denote the group of special unitary gauge transformations, and let $\G_0^\CBbb$ be its complexification.  By the results of  Hitchin, Corlette, Donaldson, and Simpson (see \cite{H,Cor,Do, Si}), we have an identification of real analytic spaces 
 $$ \X_0(\pi)\simeq \M_0(2,0):= \B^{ss}_0(2,0)\doubleslash\G_0^\CBbb $$
Combining this with recent work of Wilkin \cite{Wi} we will prove the following 
 
 \begin{theorem} \label{T:equivariant}
$ H^\ast_{eq.}(\X_0(\pi))\simeq H^\ast_{\G_0}(\B^{ss}_0(2,0))$.
 \end{theorem}
 
 Roughly speaking, this result expresses the fact that both sides compute the equivariant cohomology of a hyperk\"ahler quotient (see Section \ref{S:higgs}). As a consequence, we have from \cite[Theorem 3.2]{DWWW} the following 
 \begin{corollary}   The Poincar\'e polynomial for the $\SL(2,\CBbb)$-equivariant cohomology  is
 $$
 P_t^{\SL(2,\CBbb)}(\Hom(\pi,\SL(2,\CBbb)))= P_t^{\G_0}(\A_0^{ss}(2,0))+C(t,g)
 $$
 where
$$ P_t^{\G_0}(\A_0(2,0))=\frac{(1+t^3)^{2g}-t^{2g+2}(1+t)^{2g}}{(1-t^2)(1-t^4)}$$
 and
 \begin{align}
 C(t,g)&=
-t^{4g-4}+\frac{t^{2g+2}(1+t)^{2g}}{(1-t^2)(1-t^4)}+\frac{(1-t)^{2g}t^{4g-4}}{4(1+t^2)} \label{E:c} \\
&\qquad+ \frac{(1+t)^{2g}t^{4g-4}}{2(1-t^2)}\left(\frac{2g}{t+1}+\frac{1}{t^2-1}-\frac{1}{2} +(3-2g)\right) \notag \\
&\qquad+(1/2)(2^{2g}-1)t^{4g-2}\left( (1+t)^{2g-2}+(1-t)^{2g-2}-2 \right)\notag
 \end{align}
 \end{corollary}
\noindent In the statement above, $\A_0^{ss}(2,0)$ is the space of semistable rank 2 bundles with fixed trivial determinant, and the computation of the Poincar\'e polynomial for its $\G_0$-equivariant cohomology is in \cite{AB}.
 
 Returning to the identification in Theorem \ref{T:equivariant} and the action of the Torelli group, note that
 diffeomorphisms that do not preserve the complex structure of $M$  do not act in any natural way on $\B^{ss}_0(2,0)$.  However, 
by the contractibility of the  Teichm\"uller space of $M$ there is nevertheless a canonical  action of $\mcg(M)$ on the $\G_0$-equivariant cohomology $H^\ast_{\G_0}(\B^{ss}_0(2,0))$, and this corresponds via Theorem \ref{T:equivariant} to the action on $H^\ast_{eq.}(\X_0(\pi))$  described above (see Section \ref{S:action}).
The $\Gamma_2$ action on $\B^{ss}_0(2,0)$  given  by tensoring with $2$-torsion line bundles commutes with $\G_0^\CBbb$, and hence defines an action on $\M_0(2,0)$ and on the $\G_0$-equivariant cohomology $H^\ast_{\G_0}(\B^{ss}_0(2,0))$ of $\B^{ss}_0(2,0)$.   
 The proof of Theorem \ref{T:main} proceeds by analyzing the splitting of the corresponding long exact sequences in Morse theory over the action by $\Gamma_2$ and using the fact that this splitting is preserved by $\Tor(M)$.

The non-singular moduli space $\M_0(2,1)$ of stable Higgs bundles with a fixed determinant of  degree $1$ introduced in \cite{H} corresponds to representations of a central extension of $\pi$, and below we state the analogue of Theorem \ref{T:main} (see Theorem \ref{T:main2}).
In this case, the result essentially follows from \cite{H}, where Hitchin computed the cohomology of  $\M_0(2,1)$   using the existence of a circle action.  The perfection of the Morse-Bott function associated to the circle action follows from a result of Frankel.  

 We observe that Hitchin's method for computing the ordinary cohomology of the odd degree moduli space works as well for $\M_0(2,0)$, where the moduli space is singular.  Let $\N_0(2,k)$ denote the moduli space of semistable bundles on $M$ of rank $2$ and fixed determinant of degree $k$, and let
 \begin{equation*} 
 \R_0(\pi)=\Hom(\pi, \SU(2))/\SU(2)
 \end{equation*}
   By the result of Narasimhan-Seshadri \cite{NS} there is a real analytic equivalence 
 $
 \N_0(2,0)\simeq \R_0(\pi) $.
 We will prove 
 \begin{theorem}  \label{T:hitchin}
 The circle action on the singular variety  $\M_0(2,0)$ gives rise to a perfect Morse-Bott stratification whose minimum stratum retracts onto  $\N_0(2,0)$.  In particular, the natural inclusions 
 $\R_0(\pi)\hookrightarrow \X_0(\pi)$ and  $\R_0^{irr.}(\pi)\hookrightarrow \X_0^{irr.}(\pi)$
 induce  surjections on rational cohomology.
 \end{theorem}
 
 Here and throughout, the superscript $irr.$ stands for irreducible representations.
 A consequence of this result is a computation of the Betti numbers of $\X_0(\pi)$.  
  The Poincar\'e polynomial of $\R_0(\pi)$ was computed in \cite[Thm.\ 2.2]{CLM}.
\begin{align}
P_t(\R_0^{irr.}(\pi))
&=P_t^{\G_0}(\A_0(2,0))- \frac{ (1+t)^{2g}(1+t^2)+(1-t)^{2g}(1-t^2) }{ 2(1-t^4) } \label{E:psuirr} \\
&\qquad+\sum_{k=2}^g \left\{ {2g\choose k}-{2g\choose k-2}  \right\}
t^{2k- \epsilon(2,k)}\frac{ (1-t^{k+\epsilon(2,k)})(1-t^{2g-2k+2})}{(1-t)(1-t^4)}\notag 
\end{align}
\begin{equation} \label{E:psu}
P_t(\R_0(\pi))=P_t(\R_0^{irr.}(\pi))   -(1/2)t((1+t)^{2g}+(1-t)^{2g})
+\frac{ (1-t^{2g+2})}{(1-t)}
\end{equation}
where  $\epsilon(2,k)$ is $0$ or $1$, depending on whether $k$ is even or odd, respectively. Using Theorem \ref{T:hitchin} and  adding contributions from the other strata we obtain the following
\begin{theorem} \label{T:topology}
The  Poincar\'e polynomials of $\X_0(\pi)$ and $\X_0^{irr.}(\pi)$ are
\begin{itemize}
\item $P_t(\X_0(\pi))= P_t(\R_0(\pi)) + C(t,g) $.
\item $P_t(\X_0^{irr.}(\pi))= P_t(\R_0^{irr.}(\pi)) + C(t,g) $.
\end{itemize}
\end{theorem}

In \cite{CLM}, Cappell, Lee, and Miller also showed that the Torelli group acts non-trivially on the ordinary cohomology of $\R_0(\pi)$.  Using this  and 
the second statement  of Theorem \ref{T:hitchin}, we find  the following result, which stands in  contrast to that of Theorem \ref{T:main}.

\begin{corollary} \label{C:nontrivial}  
For $g>3$,
$\PTor(M)$ acts non-trivially on the ordinary cohomology $H^\ast(\X_0(\pi))$ and $H^\ast(\X_0^{irr.}(\pi))$.
\end{corollary}

The  action of $\Gamma_2$ on the cohomology of the moduli of space of vector bundles has been an important theme in the subject. The triviality of the action on $H^\ast(\N_0(2,1))$ 
  was first proved  in \cite[Thm.\ 1]{HN} by number theoretic methods.  It was reinterpreted by Atiyah-Bott in \cite{AB} where it is also shown that $\Gamma_2$ acts trivially on equivariant cohomology $H^\ast_{\G_0}(\A^{ss}(2,0))$ \cite[Sects. 2 and 9]{AB}.  The \emph{non-triviality} of the action of $\Gamma_2$ on $H^\ast(\M_0(2,1))$ was observed by Hitchin \cite{H} and it was further exploited in \cite{HT}.  The non-triviality of 
the action of $\Gamma_2$ on $H^\ast_{\G_0}(\B^{ss}(2,0))$ was discussed in \cite{DWWW} in connection with the failure of Kirwan surjectivity.  It follows from Theorem \ref{T:hitchin} that $\Gamma_2$ acts non-trivially on the ordinary cohomology $H^\ast(\M_0(2,0))$ as well.  
We will also prove  the following version of the result of Harder-Narasimhan for the singular moduli space (see Section \ref{S:proofs}).

\begin{theorem}  \label{T:gamma}
The action of $\Gamma_2$ on ordinary cohomology $H^\ast(\N_0(2,0))$ is trivial.
\end{theorem}

Finally, we consider the corresponding representation varieties for $\PU(2)
$ and $ \PSL(2,\CBbb)$.
Via the action of $\Gamma_2$ on the  moduli spaces $\M_0(2,0)$ and $\M_0(2,1)$ we have the following  identifications.  Let
\begin{equation*} 
\widehat \X(\pi)=\Hom(\pi,  \PSL(2,\CBbb))\doubleslash  \PSL(2,\CBbb)
\end{equation*}
Then $\widehat \X(\pi)= \widehat \X_e(\pi) \cup \widehat \X_o(\pi)$, where
\begin{equation}
 \widehat \X_e(\pi)\simeq \M_0(2,0)\bigr/\Gamma_2\ ,\ \widehat \X_o(\pi) \simeq \M_0(2,1)\bigr/\Gamma_2 \label{E:even}
\end{equation}
and the union is disjoint.
 The even component   $ \widehat X_e(\pi)$ consists of representations that lift to  $\SL(2,\CBbb)$, and  the odd component $ \widehat X_o(\pi) $ consists of representations that do not lift.  A similiar description holds for $\PU(2)$ representations: 
\begin{align} 
\widehat \R(\pi)&=\Hom(\pi,  \PU(2))\bigr/ \PU(2)= \widehat \R_e(\pi) \cup \widehat \R_o(\pi)\notag\\
 \widehat \R_e(\pi)&\simeq \N_0(2,0)\bigr/\Gamma_2\ ,\ \widehat \R_o(\pi) \simeq \N_0(2,1)\bigr/\Gamma_2\label{E:even2}
\end{align}
By considering the  $\Gamma_2$-invariant cohomology of $\M_0(2,0)$ and $\M_0(2,1)$ we deduce the following result for the action of Torelli on the space of projective representations.

\begin{corollary} \label{C:proj}
The Torelli group $\Tor(M)$ acts trivially on the  cohomology of  $\widehat \R_o(\pi)$ and 
 $ \widehat \X_o(\pi) $.  For $g>3$,  $\Tor(M)$  acts non-trivially  on the cohomology of $\widehat \R_e(\pi)$ and  $ \widehat \X_e(\pi)$.   It also acts non-trivially on the subspaces of irreducible representations.
\end{corollary}

This paper is organized as follows.  In Section \ref{S:higgs} we define the moduli spaces of bundles and Higgs bundles, state the correspondences with representation varieties, and prove the equivalence Theorem \ref{T:equivariant}.  We also discuss the results of \cite{DWWW} on equivariant Morse theory and tie this in with the $\Gamma_2$-action.  We conclude the section with the relationship between the fixed and non-fixed determinant cases.  In Section \ref{S:torelli} we show how to define the action of the Torelli group on equivariant cohomology, and using the results from Section \ref{S:higgs} we prove the main result Theorem \ref{T:main}.  We also discuss the case of odd degree.  Finally, in Section \ref{S:topology}, we prove Theorem \ref{T:hitchin} and deduce the Betti numbers of the $\SL(2,\CBbb)$ character variety. We also use this to prove the assertions of the remaining results stated above.  
Table 1 summarizes  the action of the Torelli group on the rational cohomology and equivariant cohomologies of the representation varieties  for
 $G=\SU(2)$, $\U(2)$, $\PU(2)$, $\SL(2,\CBbb)$, $\GL(2,\CBbb)$, and $\PSL(2,\CBbb)$.
 
 \medskip
\noindent \emph{Acknowledgment.}  The authors thank Bill Goldman and  Dan Margalit  for discussions.

\section{Cohomology of Higgs bundles and character varieties} \label{S:higgs}


\subsection{Definitions and equivariant cohomology} 
As in the Introduction, let $M$ be a compact Riemann surface of genus $g\geq	 2$.
Fix $p\in M$ and let ${\mathcal O}[p]$ denote the holomorphic line bundle with divisor $p$. 
Let $E\to M$ be a complex vector bundle of rank $2$ and degree $k=0,1$ and fixed hermitian metric $H$.    We denote by $\A(2,k)$ (resp.\ $\A^{ss}(2,k)$) the space of hermitian (resp.\ semistable hermitian) connections on $E$, and by 
$\B(2,k)$ (resp.\ $\B^{ss}(2,k)$) the spaces of Higgs bundles (resp.\ semistable Higgs bundles)  on $E$, i.e.\ a holomorphic bundle with a holomorphic $1$-form $\Phi$ with values in the endomorphism bundle of $E$ (the Higgs field).
The spaces $\A_0(2,k)$, $\B_0(2,k)$, $\A_0^{ss}(2,k)$, $\B_0^{ss}(2,k)$ will denote the corresponding subspaces where the induced holomorphic structure on  $\det E$ is fixed to be trivial if $k=0$, and isomorphic to ${\mathcal O}[p]$ if $k=1$, and the Higgs field is traceless. 

Let $\G$ (resp.\ $\G^\CBbb$) denote the group of real (resp.\ complex) gauge transformations acting on the spaces above by precomposition, and $\G_0$ (resp.\ $\G_0^\CBbb$) the corresponding fixed determinant groups.  We use  the following notation for the  moduli spaces of semistable bundles and semistable Higgs bundles.
\begin{align}
\begin{split} \label{E:moduli}
\N(2,k)&=\A^{ss}(2,k)\doubleslash\G^{\CBbb} \\
\N_0(2,k)&=\A_0^{ss}(2,k)\doubleslash\G_0^{\CBbb}\\
\M(2,k)&=\B^{ss}(2,k)\doubleslash\G^{\CBbb}\\
\M_0(2,k)&=\B_0^{ss}(2,k)\doubleslash\G^{\CBbb}_0
\end{split}
\end{align}
where the double slash indicates the identification of $s$-equivalent orbits.
By the results of Narasimhan-Seshadri, Hitchin, Corlette, Donaldson, and Simpson, we have the following identifications of real analytic spaces (see  \cite{NS,H,Cor,Do, Si}).
\begin{align}
\R(\pi)&:=\Hom(\pi, \U(2))\bigr/\U(2)\simeq \N(2,0)\notag  \\
\R_0(\pi)&:=\Hom(\pi, \SU(2))\bigr/\SU(2)\simeq \N_0(2,0)\label{E:reps} \\
\X(\pi)&:=\Hom(\pi, \GL(2,\CBbb))\doubleslash\GL(2,\CBbb)\simeq \M(2,0)\notag  \\
\X_0(\pi)&:=\Hom(\pi, \SL(2,\CBbb))\doubleslash\SL(2,\CBbb)\simeq \M_0(2,0) \notag 
\end{align}
where the double slash indicates the identification of orbits of reducibles with orbits of their semisimplifications.
Define the equivariant cohomologies of these spaces as in \eqref{E:equivariant}.
\begin{align} 
\begin{split} \label{E:cohomology}
H^\ast_{eq.}(\R(\pi)) &:=H^\ast_{\U(2)}(\Hom(\pi,\U(2)))   \\
H^\ast_{eq.}(\R_0(\pi)) &:=H^\ast_{\SU(2)}(\Hom(\pi,\SU(2))) \\
H^\ast_{eq.}(\X(\pi)) &:=H^\ast_{\GL(2,\CBbb)}(\Hom(\pi,\GL(2,\CBbb)))  \\
H^\ast_{eq.}(\X_0(\pi)) &:=H^\ast_{\SL(2,\CBbb)}(\Hom(\pi,\SL(2,\CBbb))) 
\end{split}
\end{align}

The construction of $\N(2,k)$ and $\N_0(2,k)$ as infinite dimensional symplectic quotient varieties is well-known (cf.\ \cite{AB,Ko}).  We briefly review the aspects of Hitchin's construction of $\M_0(2,k)$ that will be needed in the sequel (the details for $\M(2,k)$ are similar).  We furthermore focus on the case $k=0$ since that is directly related to representations of $\pi$.
We view the cotangent bundle as follows:
$$
T^\ast\A_0=\left\{ (A,\Psi) : A\in \A_0(2,0)\ ,\ \Psi\in \Omega^1(M,\sqrt{-1}\ad_0 E) \right\}
$$
where $\ad_0(E)$ denotes the bundle  of traceless skew-hermitian endomorphisms of $E$.
According to \cite{H}, $T^\ast\A_0$
 is a hyperk\"ahler manifold, and the action of the gauge group $\G_0$ has associated moment maps
\begin{equation} \label{E:momentmaps}
\mu_1(A,\Psi) = F_A+\tfrac{1}{2}[\Psi,\Psi] \ ,\
\mu_2(A,\Psi)= \sqrt{-1}\, d_A\Psi \ ,\
\mu_3(A,\Psi)= \sqrt{-1}\, d_A(\ast\Psi)
\end{equation} 
 Then $\M_0(2,0)$ is the hyperk\"ahler quotient
$$
\M_0(2,0)=\mu_1^{-1}(0)\cap \mu_2^{-1}(0)\cap \mu_3^{-1}(0) \bigr/ \G_0
$$ 
This is typically regarded as a reduction in steps in two different ways.  The first point of view (e.g.\ Hitchin and Simpson) is
\begin{equation} \label{E:hs}
\M_0(2,0)=\mu_1^{-1}(0)\cap \B_0(2,0) \bigr/ \G_0
\end{equation}
The second point of view (e.g.\ Corlette and Donaldson) is as the quotient 
\begin{equation} \label{E:cd}
\X_0(\pi)=\mu_3^{-1}(0)\cap (T^\ast\A_0)^{flat} \bigr/ \G_0
\end{equation}
where 
$$
(T^\ast\A_0)^{flat}=\left\{ (A,\Psi)\in T^\ast\A_0 : D=A+ \Psi \text{ is a flat $\SL(2,\CBbb)$ connection}\right\}
$$
In Theorem \ref{T:eq2} below, we will show that the two descriptions \eqref{E:hs} and \eqref{E:cd} give rise to the same equivariant cohomology.  

To begin, let 
\begin{equation} \label{E:bh} 
\B^H_0=\mu_1^{-1}(0)\cap \mu_2^{-1}(0)\cap \mu_3^{-1}(0)=\mu_1^{-1}(0)\cap \B_0(2,0)=\mu_3^{-1}(0)\cap (T^\ast\A_0)^{flat}
\end{equation}
denote the space of solutions to the Hitchin equations.  Let $\G_0(p)=\{g\in\G_0 : g(p)=I\}$ denote the gauge group based at the point $p$.  We denote the holonomy map
\begin{equation} \label{E:hol}
\hol_p : (T^\ast\A_0)^{flat}\bigr/\G_0(p)\lra \Hom(\pi, \SL(2,\CBbb))
\end{equation}
Note that $\hol_p$ is $\SU(2)$-equivariant with contractible fibers $\G_0^{\CBbb}(p)\bigr/\G_0(p)$.
Restricted to $\B^H_0\bigr/\G_0(p)$,  $\hol_p$ is a proper embedding.  We denote the image 
\begin{equation*} 
\Hi(M):= \hol_p\left(\B^H_0\bigr/\G_0(p) \right)\subset \Hom(\pi, \SL(2,\CBbb))
\end{equation*}
where we have included $M$ in the notation to emphasize the dependence of $\Hi(M)$  on the Riemann surface structure.  Also, note that $\Hi(M)$ consists of semisimple representations (cf.\ \cite[Thm.\ 9.13]{H}).

\begin{proposition} \label{P:retract}
The inclusion $
\Hi(M) \hookrightarrow \Hom(\pi, \SL(2,\CBbb)) 
$
is  an $\SU(2)$-equivariant deformation retract.
\end{proposition}

\begin{proof}
The proof uses the method in \cite{Cor,Do} adapted to the case of non-irreducible representations.  The idea is to use the harmonic map flow to define a flow on the space of representations.  Convergence was shown in \cite{Cor, Do}, and here we prove that in fact this defines a deformation retract.
Let $\HBbb^2$ and $\HBbb^3$ denote the 2 and 3 dimensional hyperbolic spaces, with $\pi$ acting on $\HBbb^2$ by a Fuchsian representation with quotient $M$.    Fix a lift $\tilde p$ of $p$, and a point $z\in \HBbb^3$ so that $\PU(2)$ is identified with the stabilizer of $z$ in the isometry group $\PSL(2,\CBbb)$ of $\HBbb^3$.
Given $\rho\in \Hom(\pi, \SL(2,\CBbb))$, choose  $D\in (T^\ast\A_0)^{flat}$ with $\hol_p(D)=\rho$. 
The hermitian metric  gives a unique $\rho$-equivariant lift $f:\HBbb^2\to \HBbb^3$ with $f(\tilde p)=z$.  Let $f_t$, $t\geq 0$, denote the harmonic map flow with initial condition $f$.  There is a unique continuous family $h_t\in \SL(2,\CBbb)$, $h_t^\ast=h_t$, such that $h_0=I$, and $ h_tf_t(\tilde p)=z$.  Notice that a different choice of flat connection $\widetilde D$ with $\hol_p(\widetilde D)=\rho$ will be related to $D$ by a based gauge transformation $g$.  The  flow corresponding to $\widetilde D$ is  $\tilde f_t=g\cdot f_t$, and since $g(\tilde p)=I$,  $\tilde h_t=h_t$.  Hence, $h_t$ is well-defined by $\rho$.  The flow we define is $\rho_t=h_t\rho h_t^{-1}$.

Set $\hat f_t=h_tf_t$, and 
notice  that $f_t$ is $\rho_t$-equivariant.  It follows from the Bochner formula of Eells-Sampson \cite{ES}  that the $\hat f_t$ are uniformly Lipschitz.  Hence, there is a subsequence so that
 $f_{t_j}$ converges to a harmonic map $\hat f_\infty : \HBbb^2\to\HBbb^3$ with $\hat f_\infty(\tilde p)=z$.
Moreover, $\hat f_\infty$ is equivariant with respect to some isometric action of $\pi$ on $\HBbb^3$, and this lifts to a homomorphism
$\rho_\infty : \pi\to \SL(2,\CBbb)$, with the algebraic convergence  $\rho_{t_j}\to \rho_\infty$ as $t_j\to \infty$. The harmonicity of $\hat f_\infty$ implies that a flat connection $D_\infty=A_\infty+\Psi_\infty$ with $\hol_p(D_\infty)=\rho_\infty$ satisfies $\mu_3(A_\infty, \Psi_\infty)=0$, and so $\rho_\infty\in \Hi(M)$.  We will show in the next paragraph that the limit $\rho_\infty$ is uniquely determined by $\rho$.
Hence, we have defined a map 
\begin{equation}\label{E:retraction}
r : \Hom(\pi, \SL(2,\CBbb))\lra \Hi(M)\ :\
r(\rho)=\rho_\infty=\lim_{t\to \infty}\rho_t 
\end{equation}

To prove uniqueness of the limit,    suppose $h_j= h_{t_j}$, 
$\rho_j=h_j\rho h_j^{-1}$, $\rho_j\to \sigma$,  is a convergent sequence along the flow. 
Assume first that $\rho$ is not semisimple so that $\rho$ fixes a line $L\subset \CBbb^2$.  Since the representations in $\Hi(M)$ are semisimple the $h_j$ must be unbounded, since otherwise we could extract a convergent subsequence $h_j\to h$ with $\sigma=h\rho h^{-1}$.
 Hence,  there is a sequence of unitary frames $\{v_j,w_j\}$, which we may assume converges, with respect to which 
 $$
 h_j=\left(\begin{matrix} \lambda_j &0 \\ 0&\lambda_j^{-1}\end{matrix}\right)\ , \ \lambda_j\to \infty
 $$
Fix $\alpha\in \pi$, and using the frame  $\{v_j,w_j\}$ write
$$
\rho_j(\alpha)=\left(\begin{matrix} a_j &b_j \\ c_j&d_j\end{matrix}\right)\ , \ 
\rho(\alpha)=\left(\begin{matrix} m_j &n_j \\ p_j&q_j\end{matrix}\right)
$$
Since $\rho_j=h_j\rho h_j^{-1}$ we have
\begin{equation*} 
\left(\begin{matrix} a_j &\lambda_j^2b_j \\ \lambda_j^{-2}c_j&d_j\end{matrix}\right) =
\left(\begin{matrix} m_j &n_j \\ p_j&q_j\end{matrix}\right)
\end{equation*}
and since $\rho_j$ and the frame converge whereas $\lambda_j\to \infty$, we find  $b_j\to 0$ and $p_j\to 0$.  This is true for every $\alpha\in\pi$.  Since the limit  $\sigma$ is semisimple, it must be the case that $c_j\to 0$  as well. In particular, we conclude that $\sigma$ fixes $L$, and so $\sigma$ is the just the semisimplification of $\rho$.
If $\rho$ is semisimple there exists a $\rho$-equivariant harmonic map.  While this may or may not be unique, 
using  the result of Hartman \cite{Ha}, we conclude that  the $h_j$ are bounded, and they and the associated maps converge uniquely.

 Next, we claim that the map $r$ is continuous.  Fix $\rho\in \Hom(\pi, \SL(2,\CBbb))$, and let $\rho_j\to \rho$, $\sigma_j=r(\rho_j)$, $\sigma=r(\rho)$.  Without loss of generality, we may assume the $\rho_j$ are irreducible.  Choose smoothly converging  flat connections $D_j\to D$ with $\hol_p(D_j)=\rho_j$ and $\hol_p(D)=\rho$.  Then the associated equivariant maps $f_j\to f$.  In particular, the $f_j$ have uniformly bounded energy.
 We have $\sigma_j$-equivariant harmonic maps $u_j:\HBbb^2\to\HBbb^3$ with $u_j(\tilde p)=z$.  We also have a $\sigma$-equivariant harmonic map $u:\HBbb^2\to\HBbb^3$ with $u(\tilde p)=z$.  Since the $u_j$'s have uniformly bounded energy (less than the $f_j$'s), they form a uniformly Lipschitz family of maps.  Hence, there is a subsequential limit $u_j\to \hat u$, where $\hat u$ is harmonic and equivariant with respect to some $\hat\sigma\in \Hi(M)$ and $\sigma_j\to \hat\sigma$.  We need to show $\hat \sigma=\sigma$.
 For each $\rho_j$, let $\rho_{j,t}$ denote the time $t$ flow with initial condition $\rho_j$. Define $\rho_t$ with initial condition $\rho$ similarly.  By uniqueness of the harmonic map flow, for each fixed $t$, $\rho_{j,t}\to \rho_t$ as $j\to \infty$.  Hence, we may choose a subsequence $\{j_k\}$ such that $\rho_{j_k,k}\to \sigma$.  On the other hand, since $\rho_j$ is irreducible, there exist $h_k$, $h_k^\ast=h_k$, such that $h_k\sigma_{j_k} h_k^{-1}=\rho_{j_k, k}$.
 We now consider two cases.  First, suppose the $h_k$ are bounded.  Then we may assume without loss of generality that $h_k\to h_\infty$ as $k\to \infty$, where $h_\infty$ is hermitian, and
 $h_\infty \hat\sigma h_\infty^{-1}=\sigma$.  Now $h_\infty \hat u$ and $u$ are $\sigma$-equivariant harmonic maps.   By Hartman's uniqueness theorem \cite{Ha}, either  they are  equal, or they both map to a geodesic fixed by the action of  $\sigma$.  In the former case,  $z=u(\tilde p)=h_\infty\hat u(\tilde p)=h_\infty z$, so $h_\infty$ is unitary.  But $h_\infty$ is also hermitian, so $h_\infty=I$.  In the latter case, $\sigma$ and $\hat\sigma$ are reducible.  Assuming $\sigma(\gamma)$ is not central for some $\gamma$, then $h_\infty$ carries  the orthogonal  splitting of $\hat\sigma$ to that of $\sigma$.  Hence, $h_\infty h_\infty^\ast=h_\infty^2$ is diagonal with respect to this splitting.  But then so is $h_\infty$, and hence it commutes with $\hat\sigma$.
  We conclude in either case that $\hat\sigma=\sigma$.
 If the  $h_k$ are unbounded, then  argue the same way as above.
 Namely, there is a sequence of unitary frames $\{v_k,w_k\}$, which we may assume converges, with respect to which 
 $$
 h_k=\left(\begin{matrix} \lambda_k &0 \\ 0&\lambda_k^{-1}\end{matrix}\right)\ , \ \lambda_k\to \infty
 $$
Fix $\alpha\in \pi$, and using the frame  $\{v_k,w_k\}$ write
$$
\sigma_{j_k}(\alpha)=\left(\begin{matrix} a_k &b_k \\ c_k&d_k\end{matrix}\right)\ , \ 
\rho_{j_k,k}(\alpha)=\left(\begin{matrix} m_k &n_k \\ p_k&q_k\end{matrix}\right)
$$
Since $h_k\sigma_{j_k} h_k^{-1}=\rho_{j_k}$ we have
\begin{equation*} 
\left(\begin{matrix} a_k &\lambda_k^2b_k\\ \lambda_k^{-2}c_k&d_k\end{matrix}\right) =
\left(\begin{matrix} m_k &n_k \\ p_k&q_k\end{matrix}\right)
\end{equation*}
and since $\sigma_{j_k}$ and $\rho_{j_k,k}$ converge whereas $\lambda_k\to \infty$, we find  $b_k\to 0$ and $p_k\to 0$.  This is true for every $\alpha\in\pi$.  Since the limits $\hat\sigma$ and $\sigma$ are both semisimple, it must be the case that $c_k\to 0$ and $n_k\to 0$ as well, and hence  $\hat\sigma=\sigma$.
This proves the continuity of $r$.
\end{proof}

 The following contains Theorem \ref{T:equivariant} as one case.

\begin{theorem} \label{T:eq2}
The identifications \eqref{E:reps} induce the following isomorphisms of equivariant cohomologies:
\begin{eqnarray*}
H^\ast_{eq.}(\R_0(\pi))\simeq H^\ast_{\G_0}(\A^{ss}_0(2,0)) \qquad &
H^\ast_{eq.}(\R(\pi))\simeq H^\ast_{\G}(\A^{ss}(2,0)) \\
H^\ast_{eq.}(\X_0(\pi))\simeq H^\ast_{\G_0}(\B^{ss}_0(2,0))\qquad &
H^\ast_{eq.}(\X(\pi))\simeq H^\ast_{\G}(\B^{ss}(2,0)) 
\end{eqnarray*}
\end{theorem}

\begin{proof}
We shall see below that the equivariant cohomology in the fixed and non-fixed determinant
 cases are related (see \eqref{E:tensor1}, \eqref{E:tensor2}, and Proposition \ref{P:tensor3}). 
It therefore suffices to prove the result for the fixed determinant cases.    Consider flat connections $\A_0^{flat}(2,0)$ on a rank 2 bundle with trivial determinant.  The holonomy $\hol_p$ gives   an $\SU(2)$-equivariant homeomorphism
$$
\A_0^{flat}(2,0)\bigr/\G_0(p) \simeq \Hom(\pi, \SU(2))
$$
so $H^\ast_{\G_0}(\A_0^{flat})\simeq H^\ast_{eq.}(\R_0(\pi))$.  On the other hand,
by the result in \cite{D, R}, the inclusion
$$
\A_0^{flat}(2,0)\bigr/\G_0(p)\hookrightarrow \A_0^{ss}(2,0)\doubleslash\G_0(p)
$$
is an $\SU(2)$-equivariant deformation retraction,
and so  $H^\ast_{\G_0}(\A_0^{flat})\simeq H^\ast_{\G_0}(\A^{ss}_0(2,0))$.
Since $\G_0^{\CBbb}/\G_0$ is contractible, the equivalence for $\SU(2)$ representation varieties follows.
Now consider the case of representations to $\SL(2,\CBbb)$.
  By \cite{Wi}, the inclusion
\begin{equation} \label{E:graeme}
\B^H_0\bigr/\G_0(p) \hookrightarrow \B^{ss}_0(2,0)\bigr/\G_0(p)
\end{equation} 
is an $\SU(2)$-equivariant deformation retract, so that
$H^\ast_{\G_0}(\B^{ss}_0(2,0))\simeq H^\ast_{\SU(2)}(\B^H_0\bigr/\G_0(p))$.  On the other hand, by
  Proposition \ref{P:retract}
it follows that 
$$H^\ast_{\SU(2)}(\B^H_0\bigr/\G_0(p))\simeq H^\ast_{\SU(2)}({\mathcal H}(M))\simeq H^\ast_{\SU(2)}(\Hom(\pi,\SL(2,\CBbb)))$$
Since $\SL(2,\CBbb)/\SU(2)$ and  $\G_0^{\CBbb}\bigr/\G_0$ are contractible, the result follows in this case as well.
\end{proof}

\subsection{Equivariant Morse theory} \label{S:eq} 
There is an inductive procedure to build the equivariant cohomology of $\B^{ss}(2,k)$ and $\B^{ss}_0(2,k)$, analogous to the one used in \cite{AB} for the equivariant cohomology of $\A^{ss}(2,k)$ and $\A^{ss}_0(2,k)$.  
 First, let $\mathcal C$ temporarily denote either $\A(2,k)$ or $\B(2,k)$, ${\mathcal C}_0$ either $\A_0(2,k)$ or $\B_0(2,k)$, and ${\mathcal C}^{ss}$, ${\mathcal C}_0^{ss}$ the corresponding subspaces of semistable bundles.  Note that the spaces ${\mathcal C}$, ${\mathcal C}_0$ are all contractible.  Hence, the map ${\mathcal C}\times_\G E\G\to B\G$ (resp.\ ${\mathcal C}_0\times_{\G_0} E\G_0\to B\G_0$) induces an isomorphism
 \begin{equation} \label{E:ag}
 A_\G : H^\ast(B\G)\lra H^\ast_\G({\mathcal C})\qquad \left(\text{resp.\ }A_{\G_0} : H^\ast(B\G_0)\lra H^\ast_{\G_0}({\mathcal C})\ \right)
 \end{equation}
Composing $A_\G$ (resp.\ $A_{\G_0}$) with the inclusions $\imath:{\mathcal C}^{ss}\hookrightarrow {\mathcal C}$ (resp.\ $\imath_0:{\mathcal C}_0^{ss}\hookrightarrow {\mathcal C}_0$) induces a map
\begin{align}
k_{\G} : H^\ast(B\G) &\simeq H^\ast_{\G}({\mathcal C}) \lra H^\ast_{\G}({\mathcal C}^{ss}) : k_\G=\imath^\ast\circ A_\G\label{E:kirwan} \\
\bigl(\text{resp.\ } k_{\G_0} : H^\ast(B\G_0)&\simeq H^\ast_{\G_0}({\mathcal C}_0) \lra H^\ast_{\G_0}({\mathcal C}_0^{ss}): k_{\G_0}=\imath_0^\ast\circ A_{\G_0}\ \bigr) \notag
\end{align}
We refer to the maps $k_{\G}$ and $k_{\G_0}$ as the \emph{Kirwan maps}.  When these are surjective we refer to this as \emph{Kirwan surjectivity}.
\begin{theorem}[\cite{AB}]  \label{T:ab}
Kirwan surjectivity holds for  $H^\ast_{\G}(\A^{ss}(2,k))$ and $H^\ast_{\G_0}(\A_0^{ss}(2,k))$.
\end{theorem}

The situations for $\B(2,k)$ and $\B_0(2,k)$ are somewhat different.  To describe this 
we recall the relevant results from \cite{DWWW}.  Consider the functional
$$
\YMH(A,\Psi)= \Vert F_A+\tfrac{1}{2}[\Psi,\Psi] \Vert^2_{L^2}
$$
on the space of holomorphic pairs  $\B_0(2,k)$ (resp.\ $\B(2,k)$), where the Higgs field $\Phi$ is related to $\Psi$ by $\Psi=\Phi+\Phi^\ast$.  The minimal critical set $\eta_0$ is identified with the Hitchin space (i.e.\ ${\mathcal B}_0^H$ \eqref{E:bh} for $k=0$), whereas the non-minimal critical sets $\eta_d$, $d=1,2, \ldots$, are Hitchin spaces of split bundles parametrized by the degree $d$ of the maximal destabilizing line subbundle.  Let $Y_d$ denote the stable manifold of $\eta_d$, and note that $Y_0=\B^{ss}_0(2,0)$ (resp.\ $\B^{ss}(2,0)$).
It is shown in \cite{Wi} that the $L^2$-gradient flow of $\YMH$ gives an equivariant retraction of $Y_d$ onto $\eta_d$.   Denote by $X_d=\cup_{d'\leq d} Y_{d'}$.  The main difficulty addressed in \cite{DWWW} is that unlike the situation in \cite{AB}, $Y_d$ does not have a normal bundle in $X_{d}$, and the stable manifolds $Y_d$ are singular in general.  Nevertheless, it is shown in \cite[Sect.\ 4]{DWWW} that for $\B(2,k)$ the long exact sequences of the pair $(X_d, X_{d-1})$,
\begin{equation} \label{E:xd}
\cdots\lra H^p_{\G}(X_d, X_{d-1})\buildrel\alpha^p\over \lra H^p_{\G}(X_d)\buildrel\beta^p\over\lra  H^p_{\G}(X_{d-1})\buildrel\gamma^p\over\lra\cdots
\end{equation}
split (i.e.\ $\ker \alpha^p$ is trivial) into short exact sequences for all $d$.  In particular, the Kirwan map $k_\G$ is surjective for 
$ H^\ast_{\G}(\B^{ss}(2,k))
$.


It was also shown in \cite{DWWW} that the analogous sequence \eqref{E:xd} does not split in general for 
the fixed determinant  case  $\B_0(2,k)$.  
 An explicit description of the failure of exactness goes as follows. 
 Consider the following diagram from \cite[eq.\ (27)]{DWWW}.

\begin{equation} \label{E:xdb}
\xymatrix{  & \ar[d]^{\delta^{p-1}} \vdots &&  \\
\cdots \ar[r]^{\gamma^{p-1}\qquad } & \ar[d]^{\cong} H^p_{{\G_0}}(X_d, X_{d-1}) \ar[r]^{\quad \alpha^p}
&  H^p_{{\G_0}}(X_d) \ar[r]^{\beta^p\ } &  H^p_{{\G_0}}(X_{d-1})\ar[r]^{\qquad\gamma^p\ } & \cdots \\
& \ar[d]^{\zeta^p} H^p_{{\G_0}}(\nu_d^-, \nu_d') && \\
& \ar[d]^{\lambda^p} H^p_{{\G_0}}(\nu_d^-, \nu^{\prime\prime}_d) && \\
& \ar[d]^{\delta^p}H^p_{{\G_0}}(\nu^{\prime}_d, \nu^{\prime\prime}_d) &&\\
&\vdots &&
}
\end{equation}
For our purposes,  the precise definitions of $\nu_d^-$,
 $\nu^{\prime}_d$, and 
$\nu^{\prime\prime}_d$ are not important (see \cite[Def.\ 2.1]{DWWW} for more details).
We will only use the following facts.
\begin{flalign}
H^\ast_{\G_0}(X_d,X_{d-1})&\cong H^\ast_{\G_0}(\nu_d, \nu^{\prime}_d) 
& \text{\cite[Prop.\ 3.1]{DWWW}} \\
H^\ast_{{\G_0}}(\nu_d^-, \nu^{\prime\prime}_d)\cong
H^{\ast-2\mu_d}&(\eta_d)\cong \left( H(J_0(M))\otimes H(B\U(1))\right)^{\ast-2\nu_d} &
 \text{\cite[eqs.\ (11) and (25)]{DWWW}} 
 \label{E:bb} \\
H^\ast_{{\G_0}}(\nu^{\prime}_d, \nu^{\prime\prime}_d)
&\cong
H^{\ast-2\nu_d}(\widetilde S^{2g-2-2d}M)
&\text{\cite[eq.\ (12) and Sect.\ 4.2]{DWWW}}  \label{E:bd} \\
\ker\alpha^p &\cong \ker \zeta^p &\text{\cite[Prop.\ 4.14]{DWWW}} \label{E:kernel}
\end{flalign}
where $\mu_d=g-1+2d$ and  $\widetilde\Sym^n M$ is the pull-back of the symmetric product fibration $\Sym^n M\to J_n(M)$ by the $\Gamma_2$ covering $J_n(M)\to J_n(M)$.

We now explain the relationship between this stratification and Prym representations.   First, recall the definition \eqref{E:gamma} of $\Gamma_2$.  Let 
$
\widehat \Gamma_2=\Hom(\Gamma_2, \{\pm 1\})
$.
Fixing  an homology basis $\{e_i\}_{i=1}^{2g}$ for $H_1(M,\ZBbb)$ gives a dual generating set $\{\gamma_i\}_{i=1}^{2g}$ of $\Gamma_2$ defined by $\gamma_i(e_j)=1-2\delta_{ij}$.
There  is then an isomorphism
$
 \Gamma_2\ {\buildrel \sim \over \to}\ \widehat\Gamma_2
 $ given by  $\gamma \mapsto \varphi_\gamma$ where $\varphi_\gamma(\gamma_j)=\gamma(e_j)
$.
We shall use this identification throughout the paper.
Next, using the action of $\Gamma_2$ on $\widetilde S^n M$ we have
$$
H^\ast(\widetilde\Sym^n M)=\bigoplus_{\varphi\in \widehat\Gamma_2} H^\ast(\widetilde\Sym^n M)_{\varphi}=   \bigoplus_{\varphi\in\widehat \Gamma_2}  H^\ast(\Sym^n M, \LL_{\varphi})
$$
where the subscript indicates the $\varphi$-isotypical subspace, and $\LL_{\varphi}\to S^n M$
 is the flat  line bundle determined by $\varphi$.
Let $L_\gamma\to M$ denote the flat line bundle on $M$ determined by $\gamma$.  For $\gamma\neq 1$ we have (see \cite[p.\ 98]{H})
$$
H^p(\Sym^n M, \LL_{\varphi_\gamma})=\begin{cases}  0 & p\neq n \\
\Lambda^n H^1(M, L_\gamma) & p=n
\end{cases}
$$
Now $H^1(M,L_\gamma)=W_\gamma^-$, where $W_\gamma^-$, as defined as in the Introduction, 
is the $(-1)$-eigenspace of $H^1(M_\gamma)$, where $M_\gamma$ is the double cover of $M$ defined by $\gamma$.
Using \eqref{E:bd}, we have the following

\begin{lemma}[cf.\ {\cite[Lemma 4.18]{DWWW}}] \label{L:prym}
Given a choice of homology basis there are isomorphisms
$$
H^\ast_{\G_0}(\nu_d', \nu^{\prime\prime}_d)^{\Gamma_2}=
H^{\ast-2\mu_d}(\Sym^{2g-2-2d} M) 
$$
and for $\gamma\neq 1$, 
$$
H^{6g-6-q}_{\G_0}(\nu'_d, \nu^{\prime\prime}_d)_{\varphi_\gamma}
=\begin{cases}
0  & q\neq 2g-2-2d\\
 V(q,\gamma)  & q=2g-2-2d
\end{cases}
$$
\end{lemma}

The result we will need is

\begin{proposition} \label{P:dww}
Let $S=\{2j\}_{j=1}^{g-2}$.  
\begin{enumerate}
\item 
\begin{enumerate}
\item For $q\not\in S$, the Kirwan map surjects onto $H^{6g-6-q}_{\G_0}(X_d)$ for all $d$.
\item For $q\in S$,  $p=6g-6-q$, there is precisely one $d=d_q$, $2d_q=2g-2-q$,  for which the horizontal long exact sequence in $\eqref{E:xdb}$ fails to be exact.
\end{enumerate}
\item  Let $q\in S$.
\begin{enumerate}
\item For $d>d_q$, the Kirwan map is surjective onto $H^{6g-6-q}_{\G_0}(X_{d-1})$.  
\item For $d=d_q$ we have
\begin{align*}
0\lra &\ker\gamma^{6g-6-q}\lra H^{6g-6-q}_{\G_0}(X_{d_q-1})\lra \ker \alpha^{{6g-6-q+1}}\lra 0 \\
&\quad\parallel \\
& \beta^{6g-6-q}\left(H^{6g-6-q}_{\G_0}(X_{d_q})\right)
\end{align*}
where $\ker\alpha^{{6g-6-q+1}}$ is identified with $\bigoplus_{1\neq\gamma\in\Gamma_2} V(q,\gamma)$.
\item For $d <d_q$, the sequence
$$
0\lra H^{6g-6-q}_{\G_0}(X_d, X_{d-1})  {\buildrel\quad \alpha^{6g-6-q}\over \lra} H^{6g-6-q}_{\G_0}(X_d)\ {\buildrel\quad \beta^{6g-6-q}\over \lra} H^{6g-6-q}_{\G_0}(X_{d-1})\lra 0
$$
is exact.
\end{enumerate}
\end{enumerate}
\end{proposition}

\begin{proof}
First, we claim that $\ker \alpha^{p+1}$ vanishes for all but one degree.
Since exact sequences are preserved upon restriction to isotypical pieces 
(cf.\ \cite[Lemma 4.15]{DWWW}), it suffices to prove this individually 
for $(\ker \alpha^{p+1})^{\Gamma_2}$ and $(\ker \alpha^{p+1})_{\varphi}$, $\varphi\neq 1$.
Now $(\ker \alpha^{p+1})^{\Gamma_2}$  vanishes by \cite[Cor.\ 4.27]{DWWW}.  Similarly, the result
 for $\gamma\neq 1$ is a consequence of the second statement in Lemma \ref{L:prym}.
Since $H^{6g-6-q}(\B_{d,\varepsilon}, \B_{d,\varepsilon}^{\prime\prime})\to \ker \delta^{{6g-6-q}}$ is surjective, \eqref{E:bb} implies that
 $\Gamma_2$ acts trivially on $\ker\delta^{6g-6-q}$.  Now consider the exact sequence 
 $$
 0\lra \ker\delta^{6g-6-q}\lra H^{6g-6-q}(\nu^\prime_d, \nu_d^{\prime\prime})
 \lra \ker\zeta^{{6g-6-q}+1}\lra 0
 $$
 By Lemma \ref{L:prym}  it follows that
  $$
  (\ker\zeta^{{6g-6-q}+1})_{\varphi_\gamma}
\simeq H^{6g-6-q}(\nu^\prime_d, \nu_d^{\prime\prime})_{\varphi_\gamma}\simeq  V(q,\gamma)
 $$
 Since $\ker\alpha^{{6g-6-q}+1}\simeq\ker\zeta^{{6g-6-q}+1}$ by \eqref{E:kernel}, the decomposition in part (b) follows.
This completes the proof.
\end{proof}

We conclude this section by pointing out the following

\begin{lemma} \label{L:beta}
The action of $\Gamma_2$ on $\ker \beta^p$ in \eqref{E:xdb} is trivial for all $p$.
\end{lemma}

\begin{proof}
Consider the exact sequences
$$
\xymatrix{
 0  \ar[r] & \ar[d]^{\wr} \ker\alpha^p  \ar[r] & \ar[d]^{\wr} H^p(X_d, X_{d-1})  \ar[r] & \ar[d] \ker\beta^p  \ar[r] &0 \\
 0  \ar[r] &  \ker\zeta^p \ar[r] &  H^p(\nu_d^-, \nu^\prime_d)  \ar[r] & \ker\lambda^p  \ar[r] &0 
}
$$
It follows that $\ker\beta^p\simeq \ker\lambda^p$, and this is equivariant with respect to action of $\Gamma_2$.  But
$
\ker\lambda^p\subset H^p(\nu_d, \nu^{\prime\prime}_d)  
$,
and by \eqref{E:bb} 
the $\Gamma_2$ action is trivial.
\end{proof}

\subsection{Fixed and non-fixed determinant} 
The equivariant cohomology for fixed and non-fixed determinant spaces are related through the action of
$\Gamma_2=H^1(M, \ZBbb/2)$. 

\begin{proposition}[cf.\ {\cite[Sect.\ 9]{AB}} and {\cite[Sect.\ 4.2]{DWWW}}] 
Under the action of $\Gamma_2$, we have
\begin{align}
H^\ast_{\G}(\A^{ss}(2,k))&\simeq H^\ast_{\G_0}(\A^{ss}_0(2,k))^{\Gamma_2}
 \otimes H^\ast(J_k(M))\otimes H^\ast(B\U(1))\label{E:tensor1} \\
H^\ast_{\G}(\B^{ss}(2,k))&\simeq H^\ast_{\G_0}(\B^{ss}_0(2,k))^{\Gamma_2} 
\otimes H^\ast(J_k(M))\otimes H^\ast(B\U(1)) \label{E:tensor2}
\end{align}
\end{proposition}

A similar relationship holds for the equivariant cohomology of the representation varieties.  For example, we have a  $\Gamma_2$-cover given by
$$
\Hom(\pi,\SU(2))\times \Hom(\pi,\U(1))\lra \Hom(\pi,\U(2)) : (\rho, \sigma)\mapsto \rho\cdot \sigma
$$
Moreover, the action of $\Gamma_2$ on the left commutes with conjugation by $\SU(2)$ and acts trivially on the cohomology of $\Hom(\pi,\U(1))$, so
$$
H^\ast_{\SU(2)}(\Hom(\pi,\U(2)))\simeq H^\ast_{\SU(2)}(\Hom(\pi,\SU(2)))^{\Gamma_2}\otimes H^\ast(\Hom(\pi,\U(1))
$$
This works as well for $\SL(2,\CBbb)\subset \GL(2,\CBbb)$.
Since conjugation by the center  is trivial and
$J_0(M)\simeq\Hom(\pi,\U(1))$, we conclude

\begin{proposition} \label{P:tensor3}
The following hold:
\begin{align*}
H^\ast_{eq.}(\R(\pi))&\simeq H^\ast_{eq.}(\R_0(\pi))^{\Gamma_2}\otimes H^\ast(J_0(M))\otimes H^\ast(B\U(1)) \\
H^\ast_{eq.}(\X(\pi))&\simeq H^\ast_{eq.}(\X_0(\pi))^{\Gamma_2}\otimes H^\ast(J_0(M))\otimes H^\ast(B\U(1)) 
\end{align*}
\end{proposition}

\section{Action of the Torelli group on  equivariant cohomology}  \label{S:torelli}

\subsection{General construction} 
  Let $P\to M$ be a principal bundle with compact structure group $G$.  The gauge group $\G=\Aut P$, may be regarded as the space of $G$-equivariant maps $P\to G$.  From \cite{AB} we have the following description of the classifying space of $\G$.
\begin{align}
B\G &= \Map_P(S,BG) \label{E:bg} \\
E\G &= \Map_G(P,EG) \label{E:eg}
\end{align}
where the subscript $P$ indicates the component of maps which pull-back $EG$ to $P$, and  $G$ indicates $G$-equivariant maps. 

 Let $\A_P$ denote the affine space of $G$-connections on $P$, and $T^\ast\A_P$ its cotangent space.  In the following, let ${\mathcal C}_P$ be either $\A_P$ or $T^\ast\A_P$.
 The gauge group $\G$ acts on ${\mathcal C}_P$, and the map ${\mathcal C}_P\times_\G E\G\to B\G$ induces an isomorphism as in \eqref{E:ag},
 \begin{equation} \label{E:k0}
A_\G:  H^\ast(B\G)\lra H^\ast_\G ({\mathcal C}_P)
 \end{equation}

Suppose now that $\phi: M\to M$ is a diffeomorphism with  a $G$-equivariant lift $\tilde\phi: P\to P$.   We then have the following induced maps.
\begin{align}
\phi_C &: {\mathcal C}_P\lra{\mathcal C}_P : \omega \mapsto \tilde\phi^\ast(\omega) \label{E:amap}  \\
\phi_B &: B\G \lra B\G : f \mapsto f\circ  \phi \label{E:bgmap} \\
\phi_E &: E\G \lra E\G : \tilde f \mapsto\tilde f\circ \tilde \phi \label{E:egmap}\\
\phi_\G &: \G \lra \G : g \mapsto  g\circ  \tilde\phi \label{E:gmap}
\end{align}
where in the \eqref{E:bgmap} and \eqref{E:egmap} we have used  \eqref{E:bg} and \eqref{E:eg}.  Note that $\phi_B$ gives an isomorphism on cohomology.
Combining \eqref{E:amap}, \eqref{E:egmap}, and \eqref{E:gmap},  we have  maps
\begin{align*}
\widehat\phi_C &: {\mathcal C}_P\times_\G E\G \to {\mathcal C}_P\times_\G E\G : (c,e)\mapsto (\phi_C(c), \phi_E(e)) \\
 \widehat\phi_E &: E\G\times_\G E\G \to E\G\times_\G E\G : (e_1, e_2) \mapsto (\phi_E(e_1), \phi_E(e_2))
 \end{align*}

The following result is well-known (cf.\ \cite{Fh}).  We include the proof here for the sake of completeness.
\begin{proposition} \label{P:equivariant}
For the action on cohomology we have: $A_\G\circ\phi_B^\ast=\widehat\phi_C^\ast\circ A_\G$.
\end{proposition}

\begin{proof}
The  universal connection $\Omega$ on $EG$ (cf.\ \cite{NR}) gives a  map
\begin{equation}
u: E\G = \Map_G(P,EG)\lra {\mathcal C}_P : \tilde f\mapsto  \tilde f^\ast(\Omega)\label{u}
\end{equation}
which is surjective onto $\A_P\subset{\mathcal C}_P$.
The map $u$ is clearly $\G$-equivariant.  Moreover, it induces an isomorphism on $\G$-equivariant cohomology.  Indeed, by $\G$-equivariance, the  map 
$$
\widehat u : E\G\times_\G E\G\lra {\mathcal C}_P\times_\G E\G : (e_1, e_2)\mapsto (u(e_1), \phi_E(e_2))
$$
must, by the contractibility of $E\G$ and ${\mathcal C}_P$, give an isomorphism on cohomology.
We claim that
$u\circ\phi_E = \phi_C \circ u$.
Indeed,
$$
u(\phi_E( \tilde f))=u( \tilde f\circ\tilde\phi)=( \tilde f\circ\tilde\phi)^\ast(\Omega)=\tilde\phi^\ast( \tilde f^\ast(\Omega))=\phi_C(u( \tilde f))
$$
It follows that 
\begin{equation} \label{E:1}
\hat u\circ\widehat\phi_E=\widehat\phi_C\circ \hat u
\end{equation}
On the other hand, we also have that
$\phi_E$ preserves $\G$-orbits and covers the  action $\phi_B$ of $\phi$ on $B\G$.
For if $f\in E\G$, $g\in \G$, then $\phi_E(gf)=\phi_\G(g)\phi_E(f)$.  
The map $[f]:M\to BG$ is defined $[f](s)=[f(\tilde s)]$ for any lift $\tilde s$ of $s$ to $P$.  Then the induced map on $B\G$ therefore sends $[f]$ to
$$
[\phi_E(f)](s)]=[\phi_E(f)(\tilde s)]=[f(\tilde\phi(\tilde s))]=[f(\widetilde{\phi(s)})]=[f](\phi(s))=\phi_B(f)(s)
$$
Hence, 
\begin{equation} \label{E:2}
\pi\circ\widehat\phi_E=\phi_B\circ\pi
\end{equation}
where $\pi: E\G\times_\G E\G\to B\G$ is projection to the second factor. Equations \eqref{E:1} and \eqref{E:2} imply
 $\hat u^\ast\circ A_\G\circ \phi_B^\ast=\hat u^\ast\circ \widehat\phi_A^\ast\circ A_\G$.  Since $\hat u^\ast$ is an isomorphism, the result follows.
\end{proof}

\subsection{Action on moduli spaces} \label{S:action}%

We apply the construction of the previous section to the equivariant cohomology of the moduli spaces.  First, recall that $\mcg(M)$ is  the group of components of isotopy classes of diffeomorphisms.  For $p\in M$, choose a fixed disk neighborhood $D$ of $p$.  We define $\mcg(M,D)$ to be the subgroup of the group of components of isotopy classes of diffeomorphisms that are the identity on $D$ (where the isotopies are through diffeomorphisms that are the identity on $D$).  Since any diffeomorphism has a representative fixing $D$, the forgetful map $\mcg(M,D)\to \mcg(M)$ is surjective.  We define the subgroup $\Tor(M,D)$ of the Torelli group $\Tor(M)$ similarly.
For the  trivial $\SU(2)$  bundle $P\to M$, representatives $\phi$ of elements of $\mcg(M)$ trivially lift to bundle maps $\tilde \phi$ of $P$.    For a bundle with Chern class 1, we may fix trivializations on $D$ and on the complement of $D$, and then define lifts for representatives of elements in $\mcg(M,D)$.  

Consider the space $(T^\ast\A_0)^{flat}$.  Given an element $\phi\in \Diff(M)$, the result of the previous section gave a homeomorphism $\widehat\phi$ of  $(T^\ast\A_0)^{flat}\times_{\G_0} E\G_0$. 
Recall that $(T^\ast\A_0)^{flat}$ is the space of flat $\SL(2,\CBbb)$ connections.  Then $(T^\ast\A_0)^{flat}\times_{\G_0} E\G_0$ is invariant by $\widehat\phi$.  Isotopic diffeomorphisms $\phi$ give isotopic homeomorphisms $\widehat\phi$.  Hence, we have defined an action of $\mcg(M)$ on $H^\ast_{\G_0}((T^\ast\A_0)^{flat})$.  Now
$$
H^\ast_{\G_0}((T^\ast\A_0)^{flat})=H^\ast_{\SU(2)}((T^\ast\A_0)^{flat}\bigr/\G_0(p))\simeq H^\ast_{eq.}(\X_0(\pi))
$$
where the second identification comes from the holonomy map $\hol_p$ (see \eqref{E:hol}).  
Hence, there is an action of $\mcg(M,D)$ on $H^\ast_{eq.}(\X_0(\pi))$.  
If $\phi\in\Diff(M,D)$ and $D$ is a flat connection, then $\hol_p(\phi^\ast D)=\hol_p(D)\circ\phi_\ast$, where $\phi_\ast$ denotes the action on $\pi$.
Hence,  the action of $\mcg(M, D)$  agrees with the canonical action on 
$H^\ast_{eq.}(\X_0(\pi))$ by automorphisms of $\pi$.
  In particular, automorphisms in the kernel of the surjection $\mcg(M,D)\to\mcg(M)$ act trivially, and so there is a well-defined action of $\mcg(M)$ on $\X_0(\pi)$.

Mapping classes that do not preserve a complex structure on $M$ do not act in any natural way on  $\B^{ss}_0(2,0)$.  On the other hand,  let ${\mathcal T}(M)$ denote  the Teichm\"uller space ${\mathcal T}(M)$ of $M$.  Let
\begin{align*}
\BBbb^{ss}_0(2,0)&=\bigl\{ (E,\Phi, J) :  J \text{ complex structure on } M,\\
 &\qquad\qquad (E,\Phi)\to (M,J) \text{ semistable Higgs bundle}\bigr\}
\end{align*}
Let $\Diff_0(M)$ be the group of diffeomorphisms of $M$ isotopic to the identity, with the action on  $\BBbb^{ss}_0(2,0)$ given by  pulling back.  
Projection to the $J$-factor gives a locally trivial fibration
$$
\left\{\BBbb^{ss}_0(2,0)\times E\G_0\right\}\bigr/\Diff_0(M) \ltimes \G_0 \to {\mathcal T}(M)
$$
 with fiber homeomorphic to $\B^{ss}(2,0)\times_{\G_0}E\G_0$.  By the contractibility of ${\mathcal T}(M)$,
$$
H^\ast_{\G_0}(\BBbb^{ss}_0(2,0))\simeq
H^\ast_{\G_0}(\B^{ss}_0(2,0))
$$
The mapping class group $\mcg(M)$ acts on the left hand side and so defines an action on the equivariant cohomology $H^\ast_{\G_0}(\B^{ss}_0(2,0))$.

\begin{proposition}  \label{P:mcg}
The action of $\mcg(M)$ on $H^\ast_{\G}(\B^{ss}_0(2,0))$ agrees with the action on
$H^\ast_{eq.}(\X(\pi))$. 
\end{proposition}

\begin{proof}
It suffices to check the action of $\Diff(M,D)$.  Clearly, this is
equivariant with respect to the embeddings \eqref{E:graeme} and \eqref{E:retraction}, and it commutes with the action of $\SU(2)$.  The result then follows from Theorem \ref{T:eq2}.
\end{proof}

A similar construction holds for the non-fixed determinant cohomology $H^\ast_{\G}(\B^{ss}(2,0))$. 
For $k=1$,  Notice that  the condition of determinant ${\mathcal O}[p]$, $p\in D$,  is preserved by $\Diff_0(M,D)$. 
Consider the universal space of semistable Higgs pairs:
\begin{align*}
\BBbb^{ss}_0(2,1)&=\bigl\{ (E,\Phi, J) :  J \text{ complex structure on } M,\ (E,\Phi)\to (M,J) \text{ a semistable}\\
 &\hskip1.5in  \text{Higgs bundle},\ \det E={\mathcal O}[p]\, \bigr\} 
 \end{align*}
Projection to the $J$-factor gives a locally trivial fibration
$$
\left\{\BBbb^{ss}_0(2,1)\times E\G_0\right\}\bigr/\Diff_0(M,D) \ltimes \G_0 \to \widetilde{\mathcal T}(M)
$$
 with fiber homeomorphic to $\B^{ss}_0(2,1)\times_{\G_0}E\G_0$.  Here,
 $\widetilde {\mathcal T}(M)$ fibers over ${\mathcal T}(M)$ with fiber\break $\Diff_0(M)/\Diff_0(M,D)$.  Since $\widetilde {\mathcal T}(M)$ is also contractible, the analogue of Proposition \ref{P:mcg} holds in this case as well.  A similar construction holds for the non-fixed determinant cohomology $H^\ast_{\G}(\B^{ss}(2,1))$.
 
 Clearly, the fiberwise action of $\Gamma_2$ commutes with the action of $\Tor(M)$ defined above.  Also,
 let ${\mathbb J}_0(M)\to {\mathcal T}(M)$  denote the universal Jacobian variety and 
$T^\ast_v{\mathbb J}_0(M)$ the vertical cotangent space.  The trace map 
$\overline T$ described in \cite[Sect.\ 4.2]{DWWW}  extends fiberwise to give a fibration 
 $$
 \xymatrix{
\ar[d] \left\{\BBbb^{ss}(2,0)\times E\G\right\}\bigr/\Diff_0(M,D) \ltimes \G_0 \ar[r]^{\hskip.8in \overline {\mathbb T}}& \ar[d]^{\pi} T^\ast_v{\mathbb J}_0(M) \\
{\mathcal T}(M)\ar[r]^{\sim}&{\mathcal T}(M)
}
 $$
 with fiber over $\ell$ given by
 $
 \left(\left\{\BBbb^{ss}_0(2,0)\times E\G_0\right\}\bigr/\Diff_0(M,D) \ltimes \G_0 \right) \bigr|_{\pi(\ell)}
 $.
 Then $\overline {\mathbb T}$ is  equivariant with respect to the action of $\mcg(M)$ defined above, and the action by pull-back on $T^\ast_v{\mathbb J}_0(M)$.  A similar construction holds for $k=1$.
The following is immediate (cf.\ \eqref{E:tensor2}).
 
 \begin{proposition} \label{P:decomp}
 The action of $\Tor(M)$ $($resp.\ $\Tor(M,D)$$)$ preserves the decomposition in \eqref{E:tensor2}.
 \end{proposition}
 
We now draw  some consequences from this set-up.  First, we have

\begin{proposition} \label{P:nonfixeddet}
The action of $\Tor(M)$ $($resp.\ $\Tor(M, D)$$)$ on $H^\ast_\G(\B^{ss}(2,0))\simeq H^\ast_{eq.}(\X(\pi))$ \break $($resp.\  $H^\ast_\G(\B^{ss}(2,1))$$)$ is trivial.
\end{proposition}

\begin{proof}
Let $\imath : \B^{ss}(2,0)\hookrightarrow T^\ast\A(2,0)$ denote the inclusion.
By Proposition \ref{P:equivariant}, the Kirwan map $k_\G=\imath^\ast\circ A_\G$ is equivariant.  On the other hand, by
\cite[Thm.\ 4.1]{DWWW}, it is also surjective (see the discussion following Theorem \ref{T:ab} above).  Hence, the triviality comes from the triviality of the action on the cohomology of $B\G$ (see \cite{AB}) and Proposition \ref{P:equivariant}.
\end{proof}

  The next result  follows from \eqref{E:tensor2} and  Propositions \ref{P:nonfixeddet} and \ref{P:decomp}.

\begin{corollary} \label{C:invariant}
The action of $\Tor(M)$ $($resp.\ $\Tor(M,D)$$)$ is  trivial on the $\Gamma_2$-invariant part of \break
$H^\ast_{\G_0}(\B^{ss}_0(2,0))$ $($resp.\ $H^\ast_{\G_0}(\B^{ss}_0(2,1))$$)$.
\end{corollary}

\subsection{Proof of Theorem \ref{T:main}}
As in Section \ref{S:eq}, we assume a homology basis is fixed so that we have an 
identification $\Gamma_2\simeq \widehat\Gamma_2$. 
 We may universalize the description of the stratification in Section
 \ref{S:eq} over Teichm\"uller space.  In particular,  the critical sets 
$\eta_d$, stable manifolds $X_d$, and the spaces $\nu_d^-$, $\nu^{\prime}_d$, and $\nu^{\prime\prime}_d$, 
as the complex structure of $M$ varies,
are all invariant by the action of the mapping class group described in the previous section.  Hence, as above this gives an action of the mapping class group on the equivariant cohomology of these spaces.  
Moreover, this action commutes with the action of $\Gamma_2$.  With this understood, we have the following (cf.\ Proposition \ref{P:dww}  (b)).

\begin{lemma} \label{L:alpha}
For $q\in S$ and $\gamma\in \Gamma_2$, $(\ker\alpha^{6g-6-q+1})_{\varphi_\gamma} \simeq V(q,\gamma)$ as representations of $\Tor(M)$.
\end{lemma}

Part (1) of Theorem \ref{T:main} is the precisely the statement in Corollary \ref{C:invariant}.
For part (2), 
fix $q\in S$ and  $1\neq\gamma\in\Gamma_2$. 
 Consider the horizontal long exact sequence in \eqref{E:xdb}.  By 
Proposition \ref{P:dww} (2a), \cite{AB}, and Proposition \ref{P:equivariant}, it 
follows that  $H^{6g-6-q}(X_d)_{\varphi_\gamma}=\{0\}$ for $d\geq d_q$.  
By Proposition \ref{P:dww} (2b) and Lemma \ref{L:alpha}, it follows that 
$
H^{6g-6-q}(X_{d_q-1})_{\varphi_\gamma}\simeq V(q,\gamma)
$.  Finally,  for $d< d_q$, \eqref{E:xdb} is exact by Proposition \ref{P:dww} (2c). 
 Also, by Lemma \ref{L:beta}, $\Gamma_2$ acts trivially on the image of $\alpha^{6g-6-q}$, and so

$$
H^{6g-6-q}(X_{d})_{\varphi_\gamma}\simeq H^{6g-6-q}(X_{d-1})_{\varphi_\gamma}
$$
It follows that $
H^{6g-6-q}(X_{d})_{\varphi_\gamma}\simeq V(q,\gamma)
$ for all $d< d_q$.  This proves part (2).
Finally, part (4)  follows from Proposition \ref{P:dww} (1), \cite{AB}, and Proposition \ref{P:equivariant}.

\subsection{Odd degree Hitchin space}  
Let
$0\to\ZBbb\to\tilde\pi\to\pi\to1$
be the universal central extension of $\pi=\pi_1(M)$.  In terms of a symplectic basis $\{a_i,b_i\}_{i=1}^g$, we have the following presentations.
\begin{align*}
\pi&=\Bigl<  a_i, b_i, i=1,\ldots , g : \prod_{i=1}^g[a_i, b_i]=1\Bigr> \\
\tilde\pi&=\Bigl<  a_i, b_i, c,  i=1,\ldots , g : \prod_{i=1}^g[a_i, b_i]=c \ , \text{ $c$ central}\Bigr>
\end{align*}
Set
\begin{align*} 
\R_o(\tilde \pi)&=\left\{ \rho: \tilde \pi\to \SU(2) : \rho(c)=-I\right\}\bigr/\SU(2)\\
\X_o(\tilde \pi)&=\left\{ \rho: \tilde \pi\to \SL(2,\CBbb) : \rho(c)=-I\right\}\bigr/\SL(2,\CBbb)
\end{align*}
Then we have the following identifications of smooth real analytic varieties (cf.\ \cite{AB,H}).
\begin{equation} \label{E:tilde}
\R_o(\tilde\pi)=\N_0(2,1) \ ,\
\X_o(\tilde\pi)=\M_0(2,1) 
\end{equation}
The group $\Tor(M,D)$ acts by outer automorphisms on $\pi$, and this action lifts to $\tilde\pi$.  Hence, there is an action of $\Tor(M,D)$ on $\R_o(\tilde\pi)$ and $\X_o(\tilde\pi)$, and one can verify that with respect to the identifications above this corresponds to the actions on  the equivariant cohomology of the spaces $\A^{ss}_0(2,1)$ and $\B^{ss}_0(2,1)$, respectively.

There is a free action of $\Gamma_2$ on $\R_o(\tilde\pi)$ and $\X_o(\tilde\pi)$ as before, and it is clear that the orbit of a representation under $\Gamma_2$ consists of all possible lifts of the associated projective representation.  Hence,
\begin{equation} \label{E:tildegamma}
\widehat\R_o(\pi)=\R_o(\tilde\pi)\bigr/\Gamma_2 \ ,\
\widehat\X_o(\pi)=\X_o(\tilde\pi)\bigr/\Gamma_2
\end{equation}
(see \eqref{E:even}).

Consider an unramified double cover $M_\gamma\to M$ as in the Introduction. Suppose without loss of generality that $D$ is covered by disjoint disks $\widetilde D_i$, $i=1,2$, in $M_\gamma$.  Then  for $f\in \Tor(M,D)$ there is a \emph{unique} lift to $M_\gamma$ that is the identity on the $\widetilde D_i$.  Hence, the Prym representation  \eqref{E:prym}  gives a well-defined homomorphism
$
\widetilde \Pi_\gamma : \Tor(M,D)\lra \Sp(W^-_\gamma, \ZBbb)
$.  
We also have induced representations of $\Tor(M,D)$ on \eqref{E:altprym}, now for the case where $q$ is odd as well.  As in \eqref{E:prymtorelli} we define
\begin{equation} \label{E:prymtorelli2}
\PTor(M,D)=\bigcap_{1\neq \gamma\in \Gamma_2} \ker\widetilde \Pi_\gamma
\end{equation}

With this understood, we state the following analogue of Theorem \ref{T:main}.

\begin{theorem} \label{T:main2}
 \begin{enumerate}
 \item $\Tor(M,D)$ acts trivially on $H^\ast(\X_o(\tilde\pi))^{\Gamma_2}$.
 \item For $q\in \widetilde S=\{2j-1\}_{j=1}^{g-1}$  the action of $\Tor(M)$ splits as
 $$
 H^{6g-6-q}(\X_o(\tilde\pi))=H^{6g-6-q}(\X_o(\tilde\pi))^{\Gamma_2}\oplus \bigoplus_{1\neq\gamma\in \Gamma_2} V(q,\gamma)
$$
  In particular, $\PTor(M,D)$ acts trivially and $\Tor(M,D)$ acts non-trivially  on  $H^\ast(\X_o(\tilde\pi))$. The splitting of the sum of $V(q,\gamma)$'s is canonically determined by a choice of homology basis of $M$.
 \item $\Tor(M,D)$ acts trivially on $ H^{6g-6-q}(\X_o(\tilde\pi))$ for $q\not\in \widetilde S$.
\end{enumerate}
\end{theorem}

The proof uses the stratification in \cite{DWWW} as in the even degree case.  We omit the details.


\section{Topology of the character variety and further results} \label{S:topology}


\subsection{Morse theory} 
In this section we point out that Hitchin's method for computing the cohomology of $\M_0(2,1)$ applies to the singular case $\M_0(2,0)$ as well.  Recall that  the circle action on $\M_0(2,0)$ is given by $e^{i\theta}(A,\Phi)=(A,e^{i\theta}\Phi)$.  The associated Morse function is
$$
f(A,\Phi)=2i\int_M\tr\Phi\Phi^\ast=\Vert \Phi\Vert^2_{L^2}
$$
The fixed points of the circle can be computed as in \cite[Prop.\ 7.1]{H} and correspond either to $\Phi\equiv 0$ or to splittings
$$
E=L\oplus L^\ast\ , \ \deg L=d=1,\ldots, g-1\ ,\
\Phi=\left(\begin{matrix} 0&0\\ \varphi &0\end{matrix}\right)\ , \ \varphi\in \Omega^0(M,L^2\otimes K)
$$
In particular, since the singularities of $\M_0(2,0)$ correspond to splittings
$$
E=L\oplus L^\ast\ , \ \deg L=0\ , \ 
\Phi=\left(\begin{matrix} \varphi_1&0\\ 0&\varphi_2 \end{matrix}\right)\ , \ \varphi_i\in \Omega^0(M, K)
$$
it follows that the fixed points of the $S^1$ action not in the minimum of $f$ are contained within the non-singular locus $\M_0^{s}(2,0)$ of stable Higgs bundles.  We summarize this as follows.
\begin{proposition} \label{P:singset}
The fixed point set for the circle action on $\M_0(2,0)$ is parametrized by connected components $C_d$, $d=0, \ldots, g-1$.
\begin{itemize}
\item  $C_0$ is homeomorphic to $\N_0(2,0)$.
\item $C_d$, $d=1,\ldots, g-1$,  is contained in the non-singular locus $\M_0^{s}(2,0)$.  Each $C_d$ is diffeomorphic to the $\Gamma_2$ covering $\widetilde S^{2g-2d-2}M$ of the symmetric product $S^{2g-2d-2}M$.
\end{itemize}
\end{proposition}

Recall that a rank 2 unstable holomorphic vector bundle $(E,\dbar_E)$ has a Harder-Narasimhan type $(d,-d)$, $d>0$, where $d$ is the maximal degree of a line subbundle of $E$.  We stratify $\M_0(2,0)$ by subsets
$$
U_0=\left\{ (\dbar_E,\Phi) : \dbar_E \text{ semistable}\right\} \ ,\
U_d=\left\{ (\dbar_E,\Phi) : \dbar_E \text{ has HN type }(d,-d)\right\}
$$
Setting $U_0^{s}=U_0\cap\M_0^{s}(2,0)$, $U_0^{s}$ and $U_d$, $d\geq 1$, also define a stratification of $\M_0^{s}(2,0)$.

\begin{proposition} \label{P:strat}
\begin{enumerate}
\item  $U_0$ is open in $\M_0(2,0)$ and $U_d$ is a locally closed submanifold of $\M_0^{s}(2,0)$ of real codimension $2\mu_d=2g+4d-2$, $d=1,\ldots , g-1$.  Moreover, 
$\displaystyle
\overline U_d\subset\bigcup_{d'\leq d} U_d
$.
\item  For $d=0,\ldots, g-1$, the map 
$
\Psi : \RBbb^+\times U_d\lra U_d
$
given by $\Psi_t(\dbar_E,\Phi))= (\dbar_E, e^{-t}\Phi)$ defines a deformation retract of $U_d$ to the critical set $C_d$.
\end{enumerate}
\end{proposition}

\begin{remark}
In fact, we can show more.  Namely, $\Psi$ is the gradient flow of $f$, $2\mu_d$ is the index of $f$ at the critical set $C_d$ viewed as a Morse function on $\M_0^{s}(2,0)$, and $U_d$ is the stable manifold of $\nabla f$ associated to $C_d$ (this statement may not be true for higher rank).  The proof of this is analogous to \cite[Prop.\ 4.1]{BDW} where the corresponding statements are proved for stable pairs instead of Higgs bundles. 
\end{remark}

\begin{proof}[Proof of Proposition \ref{P:strat}]
(1) follows as in \cite[Prop.\ 3.7]{D}.  To show (2), if $(\dbar_E, \Phi)\in U_0$, then $\dbar_E$ is semistable, and $(\dbar_E,e^{-t}\Phi)\to (\dbar_E,0)\in C_0$.  If $(\dbar_E, \Phi)\in U_d$, $d\geq 1$, let $L$ be a destabilizing line bundle of degree $d$, and set $\dbar_0=\dbar_L\oplus \dbar_{L^\ast}$ and write $\beta\in \Omega^{0,1}(L^2)$ for the second fundamental form of $\dbar_E$.  Also, write $\Phi=\Phi_0+\Phi_1$, where
$$
\Phi_0=\left(\begin{matrix} 0&0\\ \varphi &0\end{matrix}\right)\ , \ \varphi\in H^0(M,L^2\otimes K)
$$
and $\Phi_1$ preserves $L$.  For $g_t=\left(\begin{matrix} e^{-t/2}&0\\0&e^{t/2}\end{matrix}\right)$, we have
$$
\Psi_t(\dbar_E,\Phi)=(\dbar_E,e^{-t}\Phi)=(g_t\cdot\dbar_E,e^{-t}g_t\Phi g_{t}^{-1})=\left(\dbar_0+\left(\begin{matrix} 0&e^{-t}\beta \\0&0\end{matrix}\right), \Phi_0+\Phi_1(t)\right)
$$
and $\Phi_1(t)\to 0$ as $t\to\infty$.  It follows that $\Psi_t(\dbar_E,\Phi)\to (\dbar_0,\Phi_0)\in C_d$ as $t\to\infty$.  Since $C_d$ is fixed by $\Psi_t$, it follows that $C_d$ is a deformation retract of $U_d$.
\end{proof}

The next result is the analogue of Frankel's theorem \cite{Fr} in the context of the singular variety $\M_0(2,0)$.

\begin{theorem} \label{T:frankel}
The long exact sequence in cohomology with rational coefficients for the pairs $\left(\cup_{d\leq d_0} U_d, \cup_{d< d_0} U_d\right)$ splits into short exact sequences
$$
0\lra H^\ast\left(\cup_{d\leq d_0} U_d, \cup_{d< d_0} U_d\right)\lra  H^\ast\left(\cup_{d\leq d_0} U_d\right)\lra  H^\ast\left( \cup_{d< d_0} U_d\right)\lra 0
$$
In particular, the inclusion maps $C_d\hookrightarrow\M_0(2,0)$ induce surjections in cohomology
$$
H^\ast(\M_0(2,0))\lra H^\ast(C_d)\lra 0
$$
The same result holds for the stratification of $\M_0^{s}(2,0)$.
\end{theorem}

\begin{proof}
The proof follows the outline in \cite{AB} and \cite{Ki}.  The $S^1$-equivariant stratification $U_d$, $d=0,1,\ldots, g-1$ induces a long exact sequence in cohomology
$$
\cdots\lra
H^{q-1}_{S^1}\left( \cup_{d<d_0} U_d\right) {\buildrel \delta^{q-1}\over\lra}
H^{q}_{S^1}\left( \cup_{d\leq d_0} U_d, \cup_{d<d_0} U_d\right) {\buildrel \alpha^{q}\over\lra}
H^{q}_{S^1}\left( \cup_{d\leq d_0} U_d\right) {\buildrel \beta^{q}\over\lra}
H^{q}_{S^1}\left( \cup_{d<d_0} U_d\right) 
\lra\cdots
$$
that splits into short exact sequences due to the fact that the $S^1$-equivariant Euler class of the normal bundle of $U_d$ in $\cup_{d\leq d_0} U_d$ induces injections
$$
\alpha^q : 
H^{q}_{S^1}\left( \cup_{d\leq d_0} U_d, \cup_{d<d_0} U_d\right)\simeq H^{q-2\mu_d}_{S^1}\left( U_d\right) \simeq H^{q-2\mu_d}_{S^1}\left( C_d\right)\hookrightarrow H^{q}_{S^1}\left( \cup_{d\leq d_0} U_d\right)
$$
(cf.\ \cite[Prop.\ 13.4]{AB} and \cite[Lemma 2.18]{Ki}).  The same is true for the stratification  $U_0^{s}$, $U_d$, $d=1,\ldots, g-1$.  The point is that by Proposition \ref{P:singset} all the singularities of $\M_0(2,0)$ are contained in the open stratum, and therefore a  normal neighborhood of $U_d$ can be chosen in the smooth locus.  

Next, notice that for any $G$, the map $\jmath: X\simeq X\times EG\to X\times_G EG$ induces a map $\jmath^\ast : H^\ast_G(X)\to H^\ast(X)$.   Consider now the exact sequences 

$$
\xymatrix{ 
 & \ar[d]^{\gamma_{d_0-1}^{q-1}} 
H^{q-1}_{S^1}\left( \cup_{d<d_0} U_d\right)   
\ar[r]^{\delta^{q-1}\qquad } &  \ar[d]^{\zeta^q} 
H^{q}_{S^1}\left( \cup_{d\leq d_0} U_d, \cup_{d<d_0} U_d\right)
 \ar[r]^{\qquad\alpha^q} &\ar[d]^{\gamma_{d_0}^q} 
 H^{q}_{S^1}\left( \cup_{d\leq d_0} U_d\right)
  \ar[r]^{\beta^q} & \ar[d]^{\gamma_{d_0-1}^q}
  H^{q}_{S^1}\left( \cup_{d<d_0} U_d\right) 
 &
  \\
 & 
H^{q-1}\left( \cup_{d<d_0} U_d\right)   
\ar[r]^{\bar\delta^{q-1}\qquad } & 
H^{q}\left( \cup_{d\leq d_0} U_d, \cup_{d<d_0} U_d\right)
 \ar[r]^{\qquad\bar\alpha^q} &
 H^{q}\left( \cup_{d\leq d_0} U_d\right)
  \ar[r]^{\bar\beta^q} &
  H^{q}\left( \cup_{d<d_0} U_d\right) 
 &
}
$$
We will show by induction on $d_0$ that $\bar\alpha^q$ is injective and $\gamma_{d_0-1}^q$ is surjective.  For $d_0=0$, $\cup_{d<d_0}U_d=\emptyset$, hence both assertions trivially hold.  Assume the claim for all $d<d_0$.  If $\bar\alpha^q(\bar w)=0$, write $\bar w=\bar\delta^{q-1}(\bar u)$, where by induction $\bar u=\gamma_{d_0-1}^{q-1}(u)$ for some $u$.  If $w=\delta^{q-1}u$, then $\alpha^q w=\alpha^q\delta^{q-1}u=0$, so by the injectivity of $\alpha^q$, $w=0$, and therefore 
$$
0=\zeta^q w=\zeta^q\delta^{q-1} u=\bar\delta^{q-1}\gamma^{q-1}_{d_0-1} u=\bar\delta^{q-1}\bar u=\bar w
$$
This proves the injectivity of $\bar\alpha^q$.  To show that $\gamma_{d_0}^q$ is surjective, let $\bar w\in H^q\left(\cup_{d\leq d_0} U_d\right)$.   By the induction hypothesis for $d_0-1$, we may write $\bar\beta^q(\bar w)=\gamma^q_{d_0-1}u$.   The splitting of the top sequence implies that $u=\beta^qw$.   Therefore, 
$$
\bar\beta^q(\bar w-\gamma_{d_0}^q w)=\bar\beta^q\bar w-\gamma^q_{d_0-1}\beta^q w=\gamma_{d_0-1}^q u-\gamma^q_{d_0-1}u=0
$$
so $\bar w=\gamma^q_{d_0} w+\bar\alpha^q\bar v$.  On the other hand, under the Thom isomorphism the map $\zeta^q$ corresponds to the $t=0$ factor:
$$
\xymatrix{
\ar[d]^{\zeta^q} H^{q}_{S^1}\left( \cup_{d<d_0} U_d\right)  &\simeq &\ar[d]
 \bigoplus_{s+t=q-2\mu_d} H^s(C_d)\otimes H^t(BS^1) \\
 H^{q}\left( \cup_{d<d_0} U_d\right) &\simeq & H^{q-2\mu_d}(C_d)
}
$$
which is clearly surjective.  Hence,
$
\bar w=\gamma_{d_0}^q w +\alpha^q\zeta^q v=\gamma_{d_0}^q(w+\alpha^q v)
$,
and therefore $\gamma_{d_0}^q$ is surjective, completing the induction.
\end{proof}

\subsection{Proofs}   \label{S:proofs}
In this final section we complete the proofs of the remaining assertions in the Introduction.

\begin{proof}[Proof of Theorems \ref{T:hitchin} and \ref{T:topology}]
 Theorem \ref{T:hitchin} is an immediate consequence of Theorem \ref{T:frankel}.  To compute the Betti numbers, we have by Proposition \ref{P:singset} and Theorem \ref{T:frankel} that
\begin{align*}
P_t(\M_0(2,0))&= \sum_{d=0}^{g-1} t^{2\mu_d} P_t(C_d) 
=P_t(C_0)+ \sum_{d=1}^{g-1} t^{2\mu_d} P_t(C_d) \\
&=P_t(\N_0(2,0))+\sum_{d=1}^{g-1} t^{2\mu_d} P_t(\widetilde S^{2g-2-2d}M)
\end{align*}
The sum on the right hand side above can be evaluated (see \cite{DWWW}, eqs.\ (21), (22), and (23)).  The result is precisely $C(t,g)$ in \eqref{E:c}.   The computation  for $\M_0^{s}(2,0)$ follows similarly.  This completes the proof of Theorem \ref{T:topology}.
\end{proof}

Next, we consider Theorem \ref{T:gamma}.
  Embed $\U(1)\subset \SU(2)$ as 
$
e^{i\theta}\mapsto \left( \begin{matrix} e^{i\theta} &0\\ 0&e^{-i\theta}\end{matrix}\right)
$,
and let $N(\U(1))\subset \SU(2)$ denote its normalizer.  This induces embeddings $\A(1,0)\hookrightarrow\A_0(2,0)$ and $\G(1)\hookrightarrow \G_0$.  Similarly, the embedding $\U(1)\times\U(1)\hookrightarrow\U(2)$ as diagonal matrices induces embeddings
$\A(1,0)\times \A(1,0)\hookrightarrow \A(2,0)$ and $\G(1)\times\G(1)\hookrightarrow\G$.
With this understood, we make the following definitions.
\begin{align*}
Z_{1} &= \A(1,0)\bigr/\G_p(1) \times_{N(\U(1))} E(\SU(2)) \\
\widehat Z_{1} &= \A(1,0)\times\A(1,0)\bigr/(\G(1)\times\G(1))_p \times_{N(\U(1)\times\U(1))} E(\U(2))\\
Z_2 &= \A_0(2,0)\bigr/\G_0(p) \times_{\SU(2)} E(\SU(2))\\
\widehat Z_2 &= \A(2,0)\bigr/\G(p) \times_{\U(2)} E(\U(2))
\end{align*}
The spaces $Z_{1}$ and $Z_2$ are precisely the spaces $Y_{1b}$ and $Y_2$ in \cite[eq.\ (2.3-4)]{CLM}.  Note that $E(\SU(2))\simeq E(\U(2))$.
The trace map gives a fibration of $\widehat Z_2\to \A(1,0)/\G_p(1)$, and the inclusion
$\widehat Z_{1}\hookrightarrow \widehat Z_2$ gives a  subfibration $\widehat Z_{1}\to \A(1,0)/\G_p(1)$.  Since the  fibers over the trivial connection are $Z_{1}\hookrightarrow Z_2$, the corresponding  actions of the monodromy $\pi_1(\A(1,0)/\G_p(1))\simeq\pi_1(J_0(M))$ on the cohomology of the fibers commute with the map $H^\ast(Z_2)\to H^\ast(Z_{1})$ induced by inclusion.

\begin{lemma} \label{L:y1b}
The action of $\pi_1(\A(1,0)/\G_p(1))$ on $H^\ast(Z_{1})$ is trivial.
\end{lemma}

\begin{proof}
We have homotopy equivalences of fibrations
$$
 \xymatrix{  \ar[d] Z_{1} &\simeq&\ar[d] F=J_0(M)\times_{\ZBbb/2} E(\ZBbb/2)\\
\ar[d]^{\tr} \widehat Z_{1}   &\simeq&\ar[d]^{\det}  X=J_0(M)\times J_0(M)\times_{\ZBbb/2} E(\ZBbb/2) \\
  \A(1,0)/\G_p(1) &\simeq &J_0(M) }
 $$
where the action of $\ZBbb/2$ on the Jacobian corresponding to $Z_{1}$ is $L\mapsto L^\ast$, and on the product corresponding to $\widehat Z_{1}$ it is $(L_1,L_2)\mapsto (L_2,L_1)$.  Hence, it suffices to prove that the action of $\pi_1(J_0(M))$ on $H^\ast(F)$ from the fibration $X$ is trivial. 
Since $J_0(M)$ is a torus this will be true if the corresponding statement holds for the restriction to any embedded $S^1\subset J_0(M)$.  We may write
$$
X\bigr|_{S^1}\simeq F\times [0,1]\bigr/ (x,0)\sim (x, \phi(x))
$$
for the monodromy $\phi: F\to F$.  If $j: F\hookrightarrow X\bigr|_{S^1}$ denotes the inclusion of the fiber over $0$, there is an exact sequence
$$
\cdots\lra H^\ast(X\bigr|_{S^1}) {\buildrel j^\ast \over\lra} H^\ast(F) {\buildrel 1-\phi^\ast \over\lra} H^\ast(F)\lra H^{\ast+1}(X\bigr|_{S^1})\lra\cdots
$$
In particular,
to prove that the action of $\phi^\ast$ is trivial  it suffices to show that the inclusion $F\hookrightarrow X$ induces a surjection on cohomology.  The  $\ZBbb/2$ cover 
$$
\widetilde X =J_0(M)\times J_0(M)\times E(\ZBbb/2)
$$
is a  trivial fibration, since the fibration $J_0(M)\times J_0(M)\to J_0(M)$ given by $(L_1, L_2)\to L_1\otimes L_2$ is trivial.  
Hence, $H^\ast(\widetilde X\bigr|_{S^1})\to H^\ast(\widetilde F)$ is surjective, where the fiber $\widetilde F$ is a $\ZBbb/2$ cover of $F$.  Since the cohomology of $F$ and $X\bigr|_{S^1}$ are  the $\ZBbb/2$-invariant parts of the cohomology of $\widetilde F$ and $\widetilde X\bigr|_{S^1}$, the result follows.
\end{proof}

\begin{proof}[Proof of Theorem \ref{T:gamma}]
By \cite{AB},  $\Gamma_2$ acts trivially on $H^\ast(Z_2)\simeq H^\ast(B\G_0)$.
By Lemma \ref{L:y1b}, $\Gamma_2$ also acts trivially on  $H^\ast(Z_{1})$.  
It follows that $\Gamma_2$  acts trivially on $H^\ast(Z_2,Z_{1})$.  The proposition now follows
  by the argument in the proof of \cite[Prop.\ 3.2]{CLM}.
\end{proof}

We now have the following
\begin{corollary} \label{C:gamma}
The inclusion
 $\R_0(\pi)\hookrightarrow \X_0(\pi)$
induces a surjection
 $H^\ast\left(\X_0(\pi)\right)^{\Gamma_2}\to H^\ast\left(\R_0(\pi)\right)$
in rational cohomology.
\end{corollary}

\begin{proof}
By Theorem \ref{T:frankel}, $H^\ast\left(\X_0(\pi)\right)\to H^\ast\left(\R_0(\pi)\right)$
is surjective, and  it remains exact on the $\Gamma_2$-invariant subspaces.  The result now follows by  Theorem \ref{T:gamma}.
\end{proof}

\begin{proposition} \label{P:gl}
For $g>3$,
the Torelli group $\Tor(M)$ acts non-trivially on the rational cohomology of $\X(\pi)$ and $\R(\pi)$.
\end{proposition}

\begin{proof}
As in Section \ref{S:eq}, 
the determinant fibration $\M(2,0)\to J_0(M)$  gives a splitting
\begin{equation} \label{E:gl}
H^\ast(\M(2,0))\simeq H^\ast(\M_0(2,0))^{\Gamma_2}\otimes H^\ast(J_0(M))
\end{equation}
Indeed, the map $\M_0(2,0)\times J_0(M)\to \M(2,0)$ given by $(E,\Phi, L)\mapsto (E\otimes L, \Phi)$ is a $\Gamma_2$ covering.  Applying this construction also to $\N(2,0)$ and using  Theorem \ref{T:gamma}, we have
$$
H^\ast(\N(2,0))\simeq H^\ast(\N_0(2,0))\otimes H^\ast(J_0(M))
$$
Hence, the non-triviality of the action of $\Tor(M)$  on $H^\ast(\N(2,0))$ follows from the non-triviality of the the action on $H^\ast(\N_0(2,0))$ \cite{CLM}.  By Corollary \ref{C:gamma} and \eqref{E:gl}, this also implies the non-triviality of the action on $H^\ast(\M(2,0))$.
\end{proof}

\begin{proof}[Proof of Corollary \ref{C:nontrivial}]
By Theorem \ref{T:frankel} it suffices to show that $\PTor(M)$ acts non-trivially on the rational cohomology of the  $\R_0(\pi)$.  This is true if and only if $\PTor(M,D)$ acts non-trivially.
 By the calculation in \cite[Sect.\ 6]{CLM}, it suffices to show that the induced map 
$$
H_1(\PTor(M,D), \ZBbb)\lra H_1(\Tor(M,D), \ZBbb)
$$ 
is surjective modulo torsion.  On the other hand, by \cite[Thm.\ 2.5]{Lo} we have an exact sequence 
$$
1\lra \PTor(M,D)\lra \Tor(M,D)\lra L\lra 1
$$
where $L$ has finite index in $\prod_{0\neq \gamma\in\Gamma_2}\Sp(W_\gamma^-,\ZBbb)$ 
(note that $\Tor(M,D)$ surjects onto $\Tor(M)$).
Since $g>3$, $\dim W_\gamma^-\geq 6$.    It follows from the Matsushima vanishing theorem that $\Hom(L,\ZBbb)=\{0\}$, hence, the map $H^1(\Tor(M,D),\ZBbb)\to H^1(\PTor(M,D),\ZBbb)$ is injective.  This completes the proof.
\end{proof}

\begin{proof}[Proof of Corollary \ref{C:proj}]
By \eqref{E:tilde} and \eqref{E:tildegamma},
the first statement follows from surjectivity of the Kirwan map  for $\N_0(2,1)$ (cf.\ \cite{AB}) and Theorem \ref{T:main2} (1), respectively.  The statement for $\widehat \X_e(\pi)$ will follow by showing the corresponding statement for $\widehat\R_e(\pi)$ and using the fact, Corollary \ref{C:gamma},  that the inclusion 
$
\widehat \R_e(\pi)\hookrightarrow\widehat \X_e(\pi)
$
induces a surjection on cohomology.  Finally, since $\widehat\R_e(\pi)=\R_0(\pi)/\Gamma_2$, the result will follow from \cite{CLM}  if we can show that rationally the cohomology of $\R_0(\pi)$ is $\Gamma_2$ invariant.  But this is the content of Theorem \ref{T:gamma}.
\end{proof}

 
  \begin{table}

\begin{minipage}[b]{1\linewidth}
 
\centering
 \begin{tabular}{|| c | c | c || }
 \hline\hline
 Cohomology group &  $\Tor(M)$ acts trivially? & Reference  \\
 \hline\hline
 && \\
 $H^\ast_{eq.}(\X(\pi))$ & yes & Prop.\ \ref{P:nonfixeddet} \\ 
  &&\\
 \hline
  &&\\
  $H^\ast_{eq.}(\R(\pi))$ & yes &  \cite{AB} \\
  &&\\
 \hline\hline
  &&\\
$H^\ast(\X(\pi))$ & no & Prop.\ \ref{P:gl}  \\  
 &&\\
 \hline
  &&\\
 $H^\ast(\R(\pi))$ & no & Prop.\ \ref{P:gl} \\
&&\\
 \hline\hline
  &&\\
  $H^\ast_{eq.}(\X_0(\pi))$ & no & Thm.\ \ref{T:main} (2)  \\
  &&\\
 \hline
  &&\\
   $H^\ast_{eq.}(\R_0(\pi))$ & yes &   \cite{AB}\\
&&\\
 \hline\hline
  &&\\
  $H^\ast(\X_0(\pi))$ & no & Cor.\ \ref{C:nontrivial} \\
  &&\\
 \hline
  &&\\
    $H^\ast(\R_0(\pi))$ & no & \cite[Thm.\ 1.1]{CLM} \\
&&\\
 \hline\hline
  &&\\
  $H^\ast(\widehat\X_o(\pi))$ & yes & Cor.\ \ref{C:proj}\\
  &&\\
 \hline
  &&\\
   $H^\ast(\widehat\R_o(\pi))$ & yes & Cor.\ \ref{C:proj} \\
  &&\\
  \hline \hline
  &&\\
$H^\ast(\widehat\X_e(\pi))$ & no & Cor.\ \ref{C:proj} \\
&&\\
 \hline
  &&\\
  $H^\ast(\widehat\R_e(\pi))$ & no & Cor.\ \ref{C:proj}\\
&&\\
\hline \hline
 \end{tabular}
 \bigskip
 \caption{Action of the Torelli group on cohomology of representation varieties ($g>3$)} 
 
\footnotetext{The results in this table also apply to the cohomology of the subspaces of irreducible representations}

\end{minipage}

 \end{table}


\clearpage

\begin{thebibliography}{99}




\bibitem{AB} M.F. Atiyah and R. Bott, The Yang-Mills equations over Riemann surfaces.  Phil.\ Trans.\ R.
Soc.\ Lond.\ A  308 (1982), 523--615.

\bibitem{BDW} S. Bradlow, G. Daskalopoulos, and R. Wentworth, 
Birational equivalences of vortex moduli.  Topology  35  (1996),  no. 3, 731--748. 

\bibitem{CLM} S.E. Cappell, R. Lee, and E.Y. Miller, The action of the Torelli group on the homology of representation spaces is nontrivial. Topology 39 (2000), 851--871.

\bibitem{Cor} K.  Corlette,
Flat $G$-bundles with canonical metrics.
J. Diff. Geom. 28 (1988), 361--382.


\bibitem{CS} M. Culler and P. Shalen,  Varieties of group representations and splittings of $3$-manifolds.  Ann. of Math. (2)  117  (1983), no. 1, 109--146.


\bibitem{D} G.D. Daskalopoulos, The topology of the space of stable bundles on a Riemann surface.  J.\ Diff.\
Geom.\  36 (1992), 699--746.

\bibitem{DWWW} G.D. Daskalopoulos, J. Weitsman, G. Wilkin, and R. Wentworth, Morse theory and hyperk\"ahler Kirwan surjectivity for Higgs bundles, preprint.

\bibitem{Do} S.  Donaldson,
Twisted harmonic maps and the
self-duality equations.  Proc. London Math. Soc.   55 (1987),
127--131.

\bibitem{ES} J.  Eells and J.  Sampson,
Harmonic mappings of Riemannian manifolds.
Amer.  J.  Math.  86 (1964),
109--160.


\bibitem{Fr} T. Frankel,
Fixed points and torsion on K\"ahler manifolds.  Ann. of Math. (2)  70  (1959), 1--8.

\bibitem{Fh} C. Frohman, Unitary representations of knot groups.  Topology 
32 (1993), 121--144.


\bibitem{HN} G. Harder and M. Narasimhan, 
On the cohomology groups of moduli spaces of vector bundles on curves.
Math. Ann. 212 (1974/75), 215--248. 


\bibitem{Ha} P. Hartman, On homotopic harmonic maps. Can. J. Math. 
19 (1967), 673--687.



\bibitem{HT} T. Hausel and M. Thaddeus, Generators for the cohomology ring of the moduli space of rank 2 Higgs bundles.  Proc. London Math. Soc. (3)  88  (2004),  no. 3, 632--658. 


\bibitem{H} N. Hitchin, The self-duality equations on a Riemann surface.  Proc. London Math. Soc. (3)  55  (1987),  no. 1, 59--126. 


\bibitem{Ki} F. Kirwan, Cohomology of quotients in symplectic and algebraic geometry. Mathematical Notes, 31. Princeton University Press, Princeton, NJ, 1984.

\bibitem{Ko} S. Kobayashi, Differential geometry of complex vector bundles. Publications of the Mathematical Society of Japan, 15. Kano Memorial Lectures, 5. Princeton University Press, Princeton, NJ; Iwanami Shoten, Tokyo, 1987. 


\bibitem{Lo} E. Looijenga, Prym representations of mapping class groups.  Geom. Dedicata  64  (1997),  no. 1, 69--83.

\bibitem{LM} A. Lubotzky and A. Magid, Varieties of representations of finitely generated groups.  Mem. Amer. Math. Soc.  58  (1985),  no. 336.

\bibitem{NR} M. Narasimhan and S. Ramanan, Existence of universal connections.  Amer. J. Math.  83  1961 563--572.


\bibitem{NS} M.S. Narasimhan and C. Seshadri, Stable and unitary vector bundles on a compact Riemann
surface. Ann.\ of Math.\  82 (1965), 540-567. 



\bibitem{R} J. R\aa de, On the Yang-Mills heat equation in two 
and three dimensions. J.\ Reine.\ Angew.\
Math.\  431 (1992), 123-163.


\bibitem{Si}  C. Simpson, Constructing variations of Hodge structure using 
Yang-Mills theory and applications to uniformization.  J.\ Amer.\ Math.\ Soc.\  1 (1988), 867-918.


\bibitem{Wi} G. Wilkin,  Morse theory for the space of Higgs bundles, 
Comm. Anal. Geom.  16 (2008), 283--332.

\end{thebibliography}
\end{document}